\documentclass[a4paper,14pt, reqno]{amsart}
\usepackage{bbm}
\numberwithin{equation}{section}
\usepackage{mathdots}
\usepackage{leftidx}
\usepackage{mathrsfs}
\usepackage{amsmath,amssymb, amsthm}
\usepackage[all,poly,knot]{xy}
\usepackage[dvips]{graphicx}
\usepackage[colorlinks,linkcolor=black,anchorcolor=black,citecolor=black]{hyperref}
\usepackage{color}
\usepackage{tikz}
\usepackage{picture}
\usepackage[all]{xy}
\usepackage{enumerate}
\usepackage{makecell}
\usepackage{mathtools}
\numberwithin{equation}{section}
\newtheorem{Thm}{Theorem}[section]
\newtheorem{Lem}[Thm]{Lemma}
\newtheorem{Def}[Thm]{Definition}
\newtheorem{Cor}[Thm]{Corollary}
\newtheorem{Prop}[Thm]{Proposition}
\newtheorem{Ex}[Thm]{Example}
\newtheorem{Rem}[Thm]{Remark}

\DeclareMathOperator{\rad}{rad}
\DeclareMathOperator{\soc}{soc}

\def\lbr{\left(\begin{array}{c}}
	\def\lbrt{\left(\begin{array}{cc}}
		\def\lbrth{\left(\begin{array}{ccc}}
			\def\rbr{\end{array}\right)}

		\usepackage{mathrsfs}
		\usepackage{calc,pifont}
		\newcounter{local}
		\renewcommand\theenumi{\protect\setcounter{local}%
			{171+\the\value{enumi}}\protect\ding{\value{local}}}

\usepackage{ulem}
\usepackage[all]{xy}
\usepackage{mathdots}
\usepackage{leftidx}
\usepackage{mathrsfs}
\usepackage{amsmath,amssymb, amsthm}
\usepackage{color}
\usepackage{tikz}
\usepackage{picture}
\usepackage{enumerate}
\usepackage{makecell}
\usepackage{extarrows}
\usepackage{graphicx}
\usepackage{caption}
\usepackage{subcaption}
\usepackage{float}
\usepackage{blkarray}

\newcommand{\udots}{\mathinner{\mskip1mu\raise1pt\vbox{\kern7pt\hbox{.}}
\mskip2mu\raise4pt\hbox{.}\mskip2mu\raise7pt\hbox{.}\mskip1mu}}

\newcommand\defeqt{\stackrel{\text{$a_{i}$}}{\rule[0.6mm]{0.6cm}{0.1mm}}}
\newcommand\defeqtj{\stackrel{\text{$a_{j}$}}{\rule[0.6mm]{0.6cm}{0.1mm}}}
\newcommand\defeqi{\stackrel{\text{$i$}}{\rule[0.6mm]{0.6cm}{0.1mm}}}
\newcommand\defeqj{\stackrel{\text{$j$}}{\rule[0.6mm]{0.6cm}{0.1mm}}}
\newcommand\defeqk{\stackrel{\text{$k$}}{\rule[0.6mm]{0.6cm}{0.1mm}}}
\newcommand\defeqL{\stackrel{\text{$l$}}{\rule[0.6mm]{0.6cm}{0.1mm}}}
\newcommand\defeqjk{\stackrel{\text{$j'$}}{\rule[0.6mm]{0.6cm}{0.1mm}}}
\newcommand\defeqre{\rule[0.6mm]{0.6cm}{0.1mm}}

\keywords{}

\begin{document}
	
\title[Rogue wave]{The associated graded algebras of Brauer graph algebras II:\\ infinite representation type}
\author[Jing Guo]{Jing Guo}
\address{
	School of Mathematics and Statistics\\
	Qingdao University \\Qingdao\\
	Shandong 266071   \\
	P.R.China}
\email{jingguo@qdu.edu.cn}
\author[Yuming Liu]{Yuming Liu$^*$}
\thanks{$^*$Corresponding author.}
\address{School of Mathematical Sciences\\ Laboratory of Mathematics and Complex Systems\\ Beijing Normal University\\Beijing 100875  \\ P.R.China}
\email{ymliu@bnu.edu.cn}
\author[Yu Ye]{Yu Ye}
\address{School of Mathematical Sciences\\Wu Wen-Tsun Key Laboratory of Mathematics\\University of Science and Technology of China\\Hefei 230026   \\P.R.China \vskip-5pt and \vskip-5pt Hefei National Laboratory \\University of Science and Technology of China\\Hefei  230088 \\P.R.China  }
\email{yeyu@ustc.edu.cn}
\date{}

\begin{abstract} Let $G$ be a Brauer graph and $A$ the associated Brauer graph algebra. Denote by $\mathrm{gr}(A)$ the graded algebra associated with the radical filtration of $A$. The question when $\mathrm{gr}(A)$ is of finite representation type was answered in \cite{GL}. In the present paper, we characterize when $\mathrm{gr}(A)$ is domestic in terms of the associated Brauer graph $G$.
\end{abstract}
	
	

\maketitle \vspace{-0.9cm}
	
\section{Introduction}
	
This is a continuation of our study on the associated graded algebras of Brauer graph algebras in \cite{GL}. Since the last paper has determined the finite representation type of this class of algebras, we focus in the present paper on the infinite representation type of them. In particular, we will characterize the domestic associated graded algebras of Brauer graph algebras.

Brauer graph algebras are finite dimensional algebras and originate in the modular representation theory of finite groups. They are defined by combinatorial data based on graphs: underlying every Brauer graph algebra is a finite graph with a cyclic orientation of the edges at every vertex and a multiplicity function. The class of Brauer graph algebras coincides with the class of symmetric special biserial algebras. For the representation theory of Brauer graph algebras, we refer the reader to the survey article \cite{S}.

The idea of associating a finite dimensional algebra $A$ to a graded algebra (denoted by $\mathrm{gr}(A)$) with the radical filtration of $A$ is not rare in representation theory (see for example, \cite{CPS,RR}). For a finite dimensional algebra $A$ defined by quiver with relations, $\mathrm{gr}(A)$ often appears as a degeneration of $A$. The notion of degeneration comes from the geometric representation theory of algebras. It is known that if $\Lambda_0$ is a degeneration of some algebra $\Lambda_1$ and $\Lambda_0$ is representation-finite (resp. tame), then $\Lambda_1$ is also representation-finite (resp. tame) (see \cite{Ga,Ge}). However, the representation type of $\Lambda_0$ is usually more complicated than that of $\Lambda_1$. In \cite{GL}, we initiated the study on comparing the representation theory of $\mathrm{gr}(A)$ and that of $A$ in case that $A$ is a Brauer graph algebra. We have characterized all the algebras $\mathrm{gr}(A)$ which are of finite representation type and described the relationship between the Auslander-Reiten quivers of $\mathrm{gr}(A)$ and $A$ in this case.

A Brauer graph algebra $A$ is a self-injective (even symmetric) special biserial algebra; the associated graded algebra $\mathrm{gr}(A)$ is usually not self-injective. Nevertheless, $\mathrm{gr}(A)$ is still a special biserial algebra. Thus, both $A$ and $\mathrm{gr}(A)$ have tame representation type. To describe the tameness more precisely, one needs the notions of domestic and polynomial growth. The relationship between these notions are: domestic $\Longrightarrow$ polynomial growth $\Longrightarrow$ tame (cf. Section 2.1). Bocian and Skowro\'{n}ski have characterized when a Brauer graph algebra $A$ is domestic in \cite{BS}. In the present paper, we characterize when the associated graded algebra $\mathrm{gr}(A)$ is domestic. 

To state our main result precisely, let us first introduce some notations. 

\begin{Def}[See {\cite[Definition 2.4]{GL}}]\label{graded-degree}
Let $G$ be a Brauer graph. For each vertex $v$, we denote by $m(v)$ the multiplicity of $v$ and by $\mathrm{val}(v)$ the valency of $v$, with the convention that a loop is counted twice in $\mathrm{val}(v)$. Moreover, if $\mathrm{val}(v)=1$, we denote by $v'$ the unique vertex adjacent to $v$. For each vertex $v$ in $G$, we define the graded degree $\mathrm{grd}(v)$ as follows. 	
$$
\mathrm{grd}(v)=\begin{cases}
m(v)\mathrm{val}(v), &    \mbox{if }m(v)\mathrm{val}(v)>1,\\
m(v')\mathrm{val}(v'), &   \mbox{if }m(v)\mathrm{val}(v)=1.
\end{cases}
$$
\end{Def}

\begin{Def} [Compare with {\cite[Definition 2.12]{GL}}] \label{unbalanced-edge} Let $G$ be a Brauer graph. 
\begin{enumerate}[$(1)$]
\item If $u\defeqi v$ is an edge in $G$, we write the subgraph of $G$ by removing the edge $i$ as follows: $G\setminus i=G_{i,u}\bigcup G_{i,v}$, where $G_{i,u}$ (resp. $G_{i,v}$) is the connected branch of $G\setminus i$ containing the vertex $u$ (resp. $v$). (Note that it may happen that $G_{i,u}=G_{i,v}$.) Moreover, we denote the set of vertices in $G_{i,u}$ (resp. $G_{i,v}$) by $V(G_{i,u})$ (resp. $V(G_{i,v}))$.
\item An unbalanced edge in $G$ is defined to be an edge associated with two vertices with different graded degrees. For any unbalanced edge $v_S\defeqi v_L$  with $ \mathrm{grd}(v_S)< \mathrm{grd}(v_L)$ in $G$, we write the subgraph of $G$ by removing the edge $i$ as follows:  $G\setminus i=G_{i,L}\bigcup G_{i,S}$, where $G_{i,L}$ (resp. $G_{i,S}$) is the connected branch of $G\setminus i$ containing the vertex $v_{L}$ (resp. $v_{S}$).
\end{enumerate}
\end{Def}

\begin{Def} [See Section 3 for the details] Let $G$ be a Brauer graph.
\begin{enumerate}[$(1)$] \item A walk (for the notion of a walk in $G$, see Definition \ref{walk-in-graph} below) $v_{1}~\defeqre ~v_{2}~\defeqre ~\cdots ~\defeqre ~v_{k}$ from $v_{1}$ to $v_{k}$ in a Brauer graph $G$ is called degree decreasing if $\mathrm{grd}(v_{1})\geq \mathrm{grd}(v_{2})\geq\cdots\geq \mathrm{grd}(v_{k})$.
\item Suppose that $G$ is a Brauer tree with an exceptional vertex $v_0$ of multiplicity  $m_0$. 
\begin{enumerate}[$(2.1)$]
  \item $\kappa_0$ is defined to be the number of unbalanced edges $v_S \defeqi v_L$ in $G$ such that the exceptional vertex $v_0$ is a vertex in $G_{i,S}$.
  \item For any vertices $u$ and  $v$ in $G$, $d_G(u,v)$ is defined to be the number of edges in the unique walk from $u$ to $v$.
  \item Given two unbalanced edges $v^{(i)}_S\defeqi v^{(i)}_L$ and $v^{(j)}_S\defeqj v^{(j)}_L$ in $G$, we call
$(i,j)$ an unbalanced edge pair if $j$ is an edge in $G_{i,S}$ and $d_G(v^{(j)}_S,v^{(i)}_S)+1=d_G(v^{(j)}_L,v^{(i)}_S)$. Let $\kappa_1$ be the number of unbalanced edge pairs in $G$.
\end{enumerate}
\end{enumerate}
\end{Def}

With the notations above, we have the following main result. One interesting point in this result is that, as in the finite representation type situation (see \cite{GL}), the graded degree function plays a key role in controlling the domestic type of $\mathrm{gr}(A)$. 

\begin{Thm} {\rm(See Theorem \ref{main-result})}
Let $A$ be the Brauer graph algebra associated with a Brauer graph $G=(V(G),E(G))$ and $\mathrm{gr}(A)$ the graded algebra associated with the radical filtration of $A$, where $V(G)$ is the vertex set and $E(G)$ is the edge set. Then $\mathrm{gr}(A)$ is of polynomial growth if and ony if  $\mathrm{gr}(A)$ is domestic. 

Furthermore, we have
\begin{enumerate}[$(I)$]
\item $\mathrm{gr}(A)$ is $1$-domestic if and only if one of the following holds:
\begin{enumerate}[$(1)$]
\item $G$ is a Brauer tree with an exceptional vertex $v_0$ of multiplicity $m_0$ such that $\kappa_0(m_0-1)+\kappa_1=1$.
\item $G$ is a tree and there exist two distinct vertices $w_0,w_1$, such that the following conditions hold:
\begin{enumerate}[$(2.1)$]						
\item $m(w_0)=m(w_1)=2$ and $m(v)=1$ for $v\neq w_0,w_1$,
\item $\mathrm{grd}(w_0)=\mathrm{grd}(w_1)$,
\item Any walk from $w_0$ (or from $w_1$) is degree decreasing.
\end{enumerate}
\item $G$ is a graph with a unique cycle of odd length and $m(v)=1$ for all $v\in V(G)$, and satisfies the following conditions:
\begin{enumerate}[$(3.1)$]		
\item $\mathrm{grd}(u)=\mathrm{grd}(v)$ for any two vertices $u$ and $v$ in the unique cycle,
\item Any walk from any vertex in the unique cycle is degree decreasing.
\end{enumerate}
\end{enumerate}			
\item $\mathrm{gr}(A)$ is $2$-domestic if and only if $G$ satisfies the following conditions:
\begin{enumerate}[$(1)$]	
\item $G$ is a graph with a unique cycle of even length and $m(v)=1$ for all $v\in V(G)$,
\item $\mathrm{grd}(u)=\mathrm{grd}(v)$ for any two vertices $u$ and $v$ in the unique cycle,
\item Any walk from any vertex in the unique cycle is degree decreasing.
\end{enumerate}
\item $\mathrm{gr}(A)$ is not $n$-domestic for $n\geq3$.
\end{enumerate}
\end{Thm}

This paper is organized as follows. In Section 2, we recall various definitions and known facts needed in this paper, including representation type of finite dimensional algebras, special biserial algebras and string algebras, Brauer graph algebras and their associated graded algebras. In Section 3, we first introduce the notions of $\star$-condition and unbalanced edge pair and prove some preliminary results; then we state our main result and its consequences. The proof of main result is based on careful analyses in different cases according to the shapes of Brauer graphs; the detailed proofs and examples of three main cases are filled in Section 4-6 respectively.
	
\section{Preliminaries}

Throughout this paper, we fix an algebraically closed field $k$. Unless otherwise stated, all algebras will be finite dimensional $k$-algebras, and all their modules will be finite dimensional left modules. For a $k$-algebra $A$, we denote by $\rad(A)$ the Jacobson radical of $A$. For an $A$-module $M$, we denote by $\soc(M)$ and $\rad(M)$ the socle and the radical of $M$, respectively. The length of a module $M$ is denoted by $\ell(M)$, it means the number of composition factors in any composition series of $M$.
	
\subsection{Representation type of finite dimensional algebras} We recall the various notions on representation types of finite dimensional algebras and their relations from the textbook \cite[Section XIX.3]{SS}.

Let $A$ be a finite dimensional $k$-algebra. We say that $A$ is of finite representation type, if there are only finitely many non-isomorphic indecomposable $A$-modules.
	
Let $k[x]$ be the polynomial algebra in one variable over $k$. We say that $A$ is of tame representation type, if for any dimension $d$, there exists a finite number of $A$-$k[x]$-bimodules $Q_i$, for $1\leq i\leq n_d$, which are finitely generated and free as right $k[x]$-modules such that all but a finite number of isomorphism classes of indecomposable $A$-modules of dimension $d$ are of the form $Q_i\otimes_{k[x]}k[x]/(x-\lambda)$ for some $\lambda\in k$ and some $i$. For each $d$, let $\mu_{A}(d)$ be the least number of such $A$-$k[x]$-bimodules. We say that $A$ is of polynomial growth type if there exists a positive integer $m$ such that $\mu_{A}(d)\leq d^{m}$ for all $d\geq 2$; $A$ is of finite growth type (or equivalently, domestic) if $\mu_{A}(d)\leq m$ for some positive integer $m$ and for all  $d\geq1$ and $A$ is $n$-domestic (or $n$-parametric) if $n$ is the least such integer $m$.

Clearly every domestic algebra is of polynomial growth. In other words, if an algebra is not of polynomial growth, then the algebra is nondomestic. For examples of nondomestic algebras of polynomial growth, we refer the reader to \cite{S1}.




It is well known that an algebra of infinite representation type that is not of tame representation type is of wild representation type, however, our study does not involve the wild representation type.
	
\subsection{Special biserial algebras and string algebras} These algebras are defined by quivers and relations. For more details on these algebras, we refer to \rm\cite{BR}, \rm\cite{E}, and \rm\cite{S}.

For a quiver $Q$, we denote by $Q_0$ and $Q_1$ its vertex set and arrow set respectively. We write a path $p$ in a quiver from right to left and denote by $s(p)$ and $t(p)$ the start
and the end of $p$, respectively. The length of a path is defined in an obvious way. As usual, the trivial path at a vertex $i$ is denoted by $e_i$.
	
\begin{Def}  \label{special-biserial-algebra} A finite dimensional $k$-algebra $A$ is called special biserial if there is a
	quiver $Q$ and an admissible ideal $I$ in $kQ$ such that $A$ is Morita equivalent to $kQ/I$ and
	such that $kQ/I$ satisfies the following conditions:\\
	$(1)$ At every vertex $v$ in $Q$ there are at most two arrows starting at $v$ and there are at most
	two arrows ending at $v$.\\
	$(2)$ For every arrow $\alpha$ in $Q$, there exists at most one arrow $\beta$ such that $\beta\alpha\notin I$ and there
	exists at most one arrow $\gamma$ such that $\alpha\gamma\notin I$.
	
	A special biserial algebra $A$ is called a string algebra if the defining ideal $I$ is generated by paths.
\end{Def}

Given a special biserial algebra $A=kQ/I$, we can associate a string algebra $\bar{A}$ as follows. Set
$$L:=\{i\in Q_{0}\mid Ae_i\mbox{ is an injective and not uniserial module}\},$$
$$S_{0}:=\bigoplus\limits_{i\in L}\soc(Ae_i),$$
where $Ae_i$ denotes the indecomposable projective $A$-module at vertex $i$. Then $S_{0}$ is an ideal of $A$ and the quotient algebra $\bar{A}=A/S_{0}$ is a string algebra (cf. \cite[Section II.1.3]{E}). Note that the operation $\overline{(\cdot)}$ preserves representation type and we can reconstruct the AR-quiver of $A$ from the AR-quiver of $\overline{A}$ easily.

Suppose now that $A=kQ/I$ is a string algebra. For an arrow $\beta\in Q_{1}$, we denote by $\beta^{-1}$ the formal inverse of $\beta$ and set $s(\beta^{-1})=t(\beta)$, $t(\beta^{-1})=s(\beta)$, $(\beta^{-1})^{-1}=\beta$. For convenience, the formal inverse of an arrow will be called an inverse arrow. A word of length $n$ is defined by a sequence $c_{n}\ldots c_{2}c_{1}$, where $c_{i}\in Q_{1}$ or $c_{i}^{-1}\in Q_{1}$, and where $t(c_{i})=s(c_{i+1})$ for $1\leq i\leq n-1$. We define
$$s(c_{n}\ldots c_{2}c_{1})=s(c_{1}),~t(c_{n}\ldots c_{2}c_{1})=t(c_{n}),$$
and
$$(c_{n}\ldots c_{2}c_{1})^{-1}=c_{1}^{-1}c_{2}^{-1}\ldots c_{n}^{-1}.$$
For every vertex $v$ in $Q$, there is an empty word $1_{v}$ of length $0$ such that $t(1_{v})=s(1_{v})=v$ and $1_{v}^{-1}=1_{v}$. Suppose that a word $C:=c_{n}\ldots c_{2}c_{1}$ satisfies $s(C)=t(C)$, we define a rotation of $C$ as a word of the form $c_{i}\ldots c_{1}c_{n}\ldots c_{i+1}$. The product of two words is defined by placing them next to each other, provided that the resulting sequence is a word.

A word  $C$ is called a string provided either $C=1_{v}$ for some vertex $v$ in $Q$ or $C=c_{n}\ldots c_{2}c_{1}$ satisfying $c_{i+1}\neq c_{i}^{-1}$ for $1\leq i\leq n-1$, and no subword (or its inverse) of $C$ belongs to the ideal $I$. We say that a string $C=c_{n}\ldots c_{2}c_{1}$ with $n\geq 1$ is directed if all $c_i$ are arrows, and $C$ is inverse if all $c_i$ are inverse arrows. A string $C$ of positive length is called a band if all powers of $C$ are strings and $C$ is not a power of a string of smaller length. Note that a band must contain both arrows and inverse arrows.

On the set of strings, we consider two equivalence relations.~Firstly, $\sim$ denotes the relation which identifies $C$ and $C^{-1}$; and secondly, we define $\sim_{A}$ to be the equivalence relation which identifies each word with its rotations and their inverses. Let $\mathrm{St}(A)$ (or simply $\mathrm{St}$) be a set of representatives of strings in $A$ under $\sim$, and let $\mathrm{Ba}(A)$ (or simply $\mathrm{Ba}$) be the set of representatives of bands under $\sim_{A}$. In the following, we call a subword of a string a substring.
	
It is well known that every indecomposable module over a string algebra is either a string module or a band module. For each element $C$ in $\mathrm{St}(A)$, there is a unique string $A$-module $M(C)$ up to isomorphism. For each element $B$ in $\mathrm{Ba}(A)$ and for any finite dimensional indecomposable $k[x,x^{-1}]$-module $M=(V,\varphi)$ (where $V$ is a $n$-dimensional $k$-vector space and $\varphi$ is an invertible linear endomorphism of $V$), there is a band module $M(B,n,\varphi)$ corresponding to $B$ and $M$. For a detailed explanation of $M(B,n,\varphi)$, we refer the reader to \cite[p.160-161]{BR}.

\begin{Ex}
	Let $A=kQ$ be the Kronecker algebra defined by the following quiver
	$$\begin{picture}(8.00,20.00)
	\unitlength=1.00mm
	\special{em:linewidth 0.4pt} \linethickness{0.4pt}
	\put(5,3){$1$} \put(20,3){$2$}
	\put(9,5){\vector(1,0){8.0}}
	\put(9,3){\vector(1,0){8.0}} \put(12,6){$\alpha$}
	\put(12,-1){$\beta$}
	\end{picture}$$
	Then $A$ is a string algebra and we can choose $\mathrm{St}$ and $\mathrm{Ba}$ as follows. $$\mathrm{St}=\{1_{1},1_{2},\alpha,\beta,\beta^{-1}\alpha,\alpha\beta^{-1},\beta\alpha^{-1}\beta,\alpha\beta^{-1}\alpha,\cdots\},$$ $$\mathrm{Ba}=\{\beta^{-1}\alpha\}=\{\alpha\beta^{-1}\}.$$
   The string module $M(\beta^{-1}\alpha)$ has the Loewy diagram $\xymatrix@R=0.05pc{1\ar@{-}[dr]^{\beta}& &1\ar@{-}[dl]_{\alpha}\\ &2&}$, and the string module $M(\alpha\beta^{-1})$ has the Loewy diagram $\xymatrix@R=0.05pc{&1\ar@{-}[dl]_{\alpha}\ar@{-}[dr]^{\beta}&\\ 2&&2}$. The band module $M(\beta^{-1}\alpha,2,\varphi)$ defined by the band $\beta^{-1}\alpha$ and the $k[x,x^{-1}]$-module $(k^2,\varphi=\left
(\begin{array}{cc}
\lambda&\hspace{0.5em}1\\
0&\hspace{0.5em}\lambda
\end{array}\right))$ corresponds to the representation
$$\xymatrix{k^2 \ar@<-2pt>[r]_{\varphi^{-1}}\ar@<2pt>[r]^{I_2}& k^2,}$$ 
where $0\neq\lambda\in k$ and $I_2$ denotes the $2\times 2$ identity matrix.
\end{Ex}
	
For the representation types of special biserial algebras, there is the following theorem.
		
\begin{Thm}[{\cite[II.3.1 and II.8.1]{E}}]
\begin{enumerate}[$(1)$]
\item Any special biserial algebra $A$ is tame.
\item A string algebra $A$ is of finite representation type if and only if there is no band in $A$.
\end{enumerate}
\end{Thm}

The representation type and Auslander-Reiten quivers for self-injective special biserial algebras are well studied by Erdmann and Skowro\'{n}ski in \cite{ES}. Before stating their results, we recall some notions. For any algebra $A$, we denote by $\Gamma_{A}$ the Auslander-Reiten quiver of $A$ and by $_{s}\Gamma_{A}$ the stable Auslander-Reiten quiver of $A$. For the shapes of the translation quivers $\mathbb{Z}A_{\infty}^{\infty}$, $\mathbb{Z}A_{\infty}$, $\mathbb{Z}A_{\infty}/\textless\tau^{n}\textgreater$, $\tilde{A}_{p,q}$, we refer to \cite{HPR}. By $\tilde{A}_{p,q}$ we denote the following orientation of the quiver with underlying extended Dynkin diagram of type $\tilde{A}_{p+q-1}$:

$$\begin{small}\xymatrix@R-2em@L-1em{
&\cdot\ar@{->}[r]_{\alpha_2} &\cdot & \cdots &\cdot\ar@{->}[dr]_{\alpha_p}& \\
\cdot\ar@{->}[ur]_{\alpha_1}\ar@{->}[dr]^{\beta_1}& & & & &\cdot\\
&\cdot\ar@{->}[r]^{\beta_2} &\cdot & \cdots &\cdot\ar@{->}[ur]^{\beta_q}&
}\end{small}
$$

\begin{Thm}{\rm(\cite[Theorem 2.1]{ES})} \label{polynomial-domestic}
	Let $A=kQ/I$ be a self-injective special biserial algebra. The following are equivalent:	
	\begin{enumerate}[$(1)$]	
		\item  $_{s}\Gamma_{A}$ has a component of the form $\mathbb{Z}\tilde{A}_{p,q}$.
		\item $_{s}\Gamma_{A}$ is infinite but has no component of the form $\mathbb{Z}A_{\infty}^{\infty}$.
		\item There are positive integers $m$, $p$, $q$ such that $_{s}\Gamma_{A}$ is a disjoint union of $m$ components of the form $\mathbb{Z}\tilde{A}_{p,q}$, $m$ components of the form $\mathbb{Z}A_{\infty}/\textless\tau^{p}\textgreater$, $m$ components of the form $\mathbb{Z}A_{\infty}/\textless \tau^{q}\textgreater $ and infinitely many components of the form $\mathbb{Z}A_{\infty}/\textless\tau\textgreater$.
		\item All but a finite number of components of $\Gamma_{A}$ are of the form $\mathbb{Z}A_{\infty}/\textless\tau\textgreater$.
		\item The number of primitive walks in $A$ is a positive integer.
		
		\item $A$ is representation-infinite domestic.
		\item $A$ is representation-infinite of polynomial growth.
	\end{enumerate}
\end{Thm}
\begin{Thm}{\rm(\cite[Theorem 2.2]{ES})} \label{not-of-polynomial}
	Let $A=kQ/I$ be a self-injective special biserial algebra. The following are equivalent:
	\begin{enumerate}[$(1)$]
		\item $_{s}\Gamma_{A}$ has a component of the form $\mathbb{Z}A_{\infty}^{\infty}$.
		\item $_{s}\Gamma_{A}$ has infinitely many (regular) components of the form $\mathbb{Z}A_{\infty}^{\infty}$.
		\item $_{s}\Gamma_{A}$ is a disjoint union of a finite number of components of the form $\mathbb{Z}A_{\infty}/\textless\tau^{n}\textgreater$ with $n>1$, infinitely many components of the form $\mathbb{Z}A_{\infty}/\textless\tau\textgreater$ and infinitely many components of the form $\mathbb{Z}A_{\infty}^{\infty}$.
		\item $A$ has infinitely many primitive walks.
		\item $A$ is not of polynomial growth.
	\end{enumerate}
\end{Thm}

\begin{Rem} For the definition of primitive walks (= primitive $V$-sequences) in a special biserial algebra $A$, we refer to \cite[Section 2]{WW}. In fact, the primitive walks in $A$ are defined using the associated string algebra $\overline{A}$ and precisely correspond to the bands in $\overline{A}$.
\end{Rem}

\begin{Ex}
	Let $A=kQ/I$ be the self-injective special biserial algebra defined by the following quiver	

	$$
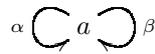
\xymatrix@r{
		a\ar@(ul,dl)_{\alpha}\ar@(ur,dr)^{\beta}
	}$$
	\\
	and the admissible ideal $I$ generated by $\alpha^2$, $\beta^2$ and $\alpha\beta-\beta\alpha$. We can choose $\mathrm{Ba}$ for $\overline{A}$ as follows. $$\mathrm{Ba}=\{\beta^{-1}\alpha\}=\{\alpha\beta^{-1}\}.$$
	Let $Q'$ and $Q''$ be the following quivers respectively:
	$$\xymatrix@R-1em{	1	\ar@{->}@<2pt>[r]^{\zeta_1} \ar@{->}@<-2pt>[r]_{\zeta_2} &2 };\quad \quad \quad \quad 
\xymatrix@R-1em{	1	 &2 \ar@{->}@<-2pt>[l]_{\zeta_1} \ar@{->}@<2pt>[l]^{\zeta_2}}.$$
	We have two quiver homomorphisms $u$ and $u'$ from $Q'$ to $Q$ as follows.
	$$u(1)=u(2)=a, ~u(\zeta_1)=\alpha,~ u(\zeta_2)=\beta;
	u'(1)=u'(2)=a, ~u'(\zeta_1)=\beta,~ u'(\zeta_2)=\alpha.$$
	By \cite[Section 2]{WW}, we have that $u$ and $u'$ are primitive walks in $A$. They correspond to the bands  $\beta^{-1}\alpha$ and $\alpha^{-1}\beta$ in $\overline{A}$  respectively. Similarly, we also have two quiver homomorphisms from $Q''$ to $Q$ and they are primitive walks in $A$ which correspond to the bands  $\beta\alpha^{-1}$ and $\alpha\beta^{-1}$ in $\overline{A}$. According to \cite[Proposition 2.3]{WW}, all the above primitive walks are equivalent and define isomorphic band modules over $\overline{A}$.
\end{Ex}

\subsection{Brauer graph algebras and their associated graded algebras} In this subsection, we briefly recall some notions and results on Brauer graphs, Brauer graph algebras and their associated graded algebras. For more details and examples, we refer to \cite[Section 2]{S} and \cite[Section 2]{GL}.
	
Recall that a Brauer graph is denoted by $G=(V(G),E(G),m,\mathfrak{o})$, where $V(G)$ is the vertex set, $E(G)$ is the edge set, $m$ is the multiplicity function, and $\mathfrak{o}$ is the orientation of $G$. For simplicity, we often leave out the symbol $\mathfrak{o}$. In case that $G$ is a Brauer tree, the exceptional vertex and whose multiplicity will be denoted by $v_0$ and $m_{0}$, respectively.
	
The Brauer graph algebra $A$ associated with a Brauer graph $G=(V(G),E(G),m)$ has the form $kQ/I$, where the vertex set $Q_0$ of $Q$ is identified with the edge set $E(G)$ of $G$, and the arrow set $Q_1$ of $Q$ is determined by the orientation of $G$. Note that there are at most two arrows starting and ending at every vertex of $Q$. Every vertex $v\in V(G)$ such that $m(v)\mathrm{val}(v)\geq 2$ (i.e. $v$ is not truncated), gives rise to a unique cycle $C_v$ in $Q$, called a special cycle at $v$. If $G$ contains no loops, then any special cycle in $Q$ is a simple cycle (i.e. a cycle with no repeated arrows and no repeated vertices). Let $C_v$ be such a special cycle at $v$. Then if $C_v$ is a representative in its cyclic permutation class such that $t(C_v)=i=s(C_v)$, $i\in Q_0$, we say that $C_v$ is a special $i$-cycle at $v$. If a special $i$-cycle at $v$ has starting arrow $\alpha$, then we denote this special $i$-cycle at $v$ by $C_v(\alpha)$. Note that if $i\in E(G)$ is not a loop, then the special $i$-cycle at $v$ is unique and we simply write it by $C_v$.
		
The ideal $I$ in $kQ$ is generated by three types of relations:
	
Relation of the first type:
$${C_v(\alpha)}^{m(v)}-{C_{v'}(\alpha')}^{m(v')},$$
for any $i\in Q_0$ and for any special $i$-cycles $C_v(\alpha)$ and $C_{v'}(\alpha')$ at $v$ and $v'$ respectively such that both $v$ and $v'$ are not truncated.
	
Relation of the second type:
$$\alpha{C_v(\alpha)}^{m(v)},$$
for any $i\in Q_0$, any $v\in V(G)$ and where $C_v(\alpha)$ is a special $i$-cycle at $v$ with starting arrow $\alpha$.
	
Relation of the third type:
$$\beta\alpha,$$
for any $\alpha,\beta\in Q_1$ such that $\beta\alpha$ is not a subpath of any special cycle except if $\beta=\alpha$ is a loop associated with a vertex $v$ of valency one and multiplicity $m(v)>1$.
	
It is well known that Brauer graph algebras coincide with symmetric special biserial algebras. From this point of view, Bocian and Skowro\'{n}ski give a characterization of the domestic Brauer graph algebras in \cite{BS}.

\begin{Thm}[See {\cite{BS}} and \cite{ES}, or see {\cite[Theorem 5.1]{S}}] \label{Brauer-graph-domestic}
	Let $A$ be the Brauer graph algebra with a  Brauer graph $G=(V(G), E(G),m)$, where $V(G)$ is the vertex set, $E(G)$ is the edge set and $m$ is the multiplicity function of $G$. Then
	\begin{enumerate}[$(a)$]
		\item $A$ is $1$-domestic if and only if one of the following holds
		\begin{enumerate}[$(1)$]
			\item $G$ is a tree with $m(v)=2$ for exactly two vertices $v=w_0,w_1\in V(G)$ and $m(v)=1$ for all $v\in V(G)$, $v\neq w_0,w_1$.
			\item $G$ is a graph with a unique cycle of odd length and $m(v)=1$ for all $v\in V(G)$.
		\end{enumerate}
		\item $A$ is $2$-domestic if and only if $G$ is a graph with a unique cycle of even length and $m(v)=1$ for all $v\in V(G)$.
		\item There are no $n$-domestic Brauer graph algebras for $n\geq3$.
	\end{enumerate}
\end{Thm}

Note that if $G$ is neither one of the above mentioned cases in Theorem \ref{Brauer-graph-domestic} nor a Brauer tree, then, by Theorem \ref{polynomial-domestic}, the corresponding Brauer graph algebra $A$ is not of polynomial growth.
	
We now turn to the associated graded algebras of Brauer graph algebras. The notion of the graded algebra (denoted by $\mathrm{gr}(A)$) associated to a finite dimensional algebra $A$  with the radical filtration of $A$ plays an important role in the representation theory. For the definition and elementary properties of $\mathrm{gr}(A)$, we refer to \cite[Subsection 2.2]{GL}. Recall from \cite[Subsection 2.3]{GL} that, for a Brauer graph algebra $A=kQ/I$ associated with a Brauer graph $G$, the graded algebra $\mathrm{gr}(A)$ (associated with the radical filtration) of $A$ has the same dimension with $A$ and can be described by the same quiver and some modified relations. More precisely, $\mathrm{gr}(A)=kQ/I'$, where $I'$ is an admissible ideal in $kQ$ generated by relations of the second and the third types in $I$ and modified relations of the first type in $I$. For a relation of the first type  ${C_v(\alpha)}^{m(v)}-{C_{v'}(\alpha')}^{m(v')}$ in $I$, its modified relation is defined by the term of smaller length between ${C_v(\alpha)}^{m(v)}$ and ${C_{v'}(\alpha')}^{m(v')}$.

From the above description, we know that $\mathrm{gr}(A)$ is also special biserial (but not necessarily self-injective). Thus we can reduce the study on the representation types of $A$ and $\mathrm{gr}(A)$ to that of their associated string algebras $\overline{A}$ and $\overline {\mathrm{gr}(A)}$. The string algebra $\overline{A}$ is defined by

\begin{equation}
\overline{A}=A/(\bigoplus\limits_{i\in L}\soc(Ae_i)),
\end{equation}
where $$L=\{i\in Q_{0}|\rad(Ae_{i})/\soc(Ae_{i})=V_{i,1}\oplus V_{i,2}, V_{i,1}\neq 0, V_{i,2}\neq 0\}.$$
For each $i\in L$, there is a relation $\rho_i=p_i-q_i$ of the first type in $I$, where the length of $p_i$ is $\ell(V_{i,1})+1$, the length of $q_i$ is $\ell(V_{i,2})+1$. Therefore $\overline{A}$ can be described by the same quiver $Q$ and an admissible ideal $I_1$ in $kQ$, where $I_1$ is generated by the ideal $I$ and new relations $\{p_i, q_i\mid i\in L\}$. Similarly, the string algebra $\overline {\mathrm{gr}(A)}$ is defined by
\begin{equation}
\overline {\mathrm{gr}(A)}=\mathrm{gr}(A)/(\bigoplus\limits_{i\in L'} \soc(\mathrm{gr}(A)e_{i})),
\end{equation}
where $$L'=\{i\in L|\ell(V_{i,1})=\ell(V_{i,2})\}.$$
Note that for each $i\in L'$, there is a relation $\rho_i=p_i-q_i$ in $I'$ such that $p_i$ and $q_i$ have the same length. Therefore $\overline {\mathrm{gr}(A)}$ can be described by the same quiver $Q$ and an admissible ideal $I_2$ in $kQ$, where $I_2$ is generated by the ideal $I'$ and new relations $\{p_i, q_i\mid i\in L'\}$.

As a conclusion, the four concerned algebras have the same quiver and the following displayed formulas:
$$A=kQ/I,\quad \overline{A}=kQ/I_1,\quad \mathrm{gr}(A)=kQ/I',\quad \overline{\mathrm{gr}(A)}=kQ/I_2.$$
In order to describe some relationships among these algebras, we further generalize some notions from \cite{GL}.

\begin{Def}[Compare with {\cite[Definition 2.12]{GL}}] \label{unbalanced-edge} Let $G=(V(G), E(G), m)$ be a Brauer graph with graded degree function $\mathrm{grd}$ and $A=kQ/I$ the corresponding Brauer graph algebra. We identify $Q_0$ with $E(G)$ by the natural bijection between them.
\begin{enumerate}[$(1)$]
\item For an unbalanced edge $u\defeqi v$ in $G$, we denote the endpoints of $i$ by $v_L^{(i)}$, $v_S^{(i)}$ with $\mathrm{grd}(v_L^{(i)})>\mathrm{grd}(v_S^{(i)})$. Whenever the context is clear we will omit the superscript $(i)$. Moreover, we define
\begin{equation}
n_i =\mbox{ the number of edges in }G_{i,S},
\end{equation}
where $G_{i,S}$ is the connected branch of $G\setminus i$ containing the vertex $v_{S}$.
\item For an unbalanced edge $v_S\defeqi v_L$ in $G$, there is a relation of the first type $\rho_i=p_i-q_i$ in $I$, where $p_i=C_{v_S}^{m(v_S)}$,   $q_i=C_{v_L}^{m(v_L)}$ are paths with lengths $\mathrm{grd}(v_{S})$, $\mathrm{grd}(v_{L})$ respectively. We define the following sets:
\begin{equation}
\mathbb{W}=\{i\in Q_{0}\vert \rad(A e_{i})/\soc(A e_{i})=V_1 \oplus V_2, V_1\neq0, V_2\neq0,\ell(V_1)\neq\ell(V_2)\}\subseteq Q_0,
\end{equation}
\begin{equation}
\mathbb{P}=\bigcup\limits_{i\in \mathbb{W}}\{r_i|r_i\mbox{ is the longer path between }p_i\mbox{ and }q_i\}.
\end{equation}
\end{enumerate}
\end{Def}

Note that the set of unbalanced edges is identified with $\mathbb{W}$ under the natural bijection between $Q_0$ and $E(G)$, and that $s(r_i)=t(r_i)=i$ for $r_{i}\in \mathbb{P}$.	

By the definitions of $\overline{A}$ and $\overline{gr(A)}$, we have that $\overline{A}$ is a quotient algebra of $\overline{\mathrm{gr}(A)}$, that is, $\overline{A}\cong\overline{\mathrm{gr}(A)}/I_3$, where the ideal $I_3$ is the $k$-vector space with basis given by the paths in the set $\mathbb{P}$. In particular, any indecomposable $\mathrm{gr}(A)$-module that is not $\overline{\mathrm{gr}(A)}$-module is an indecomposable projective-injective $\mathrm{gr}(A)$-module.

For convenience, we record displayed formulas of the ideals $I,I',I_1, I_2, I_3$ in $kQ$ (see \cite[Subsection 3.2]{GL}):
$$R_1:=\{\mbox{Relation of the first type in }I\},\quad I_0:=\langle \mbox{Relation of the second type or the third type in }I\rangle;$$
$$I=I_0+\langle R_1\rangle;$$
$$I'=I_0+\langle p_i-q_i\in R_1 \mid i\in Q_0, i\notin\mathbb{W}\rangle+\langle q_i\mid i\in \mathbb{W}, p_i-q_i\in R_1, q_i\mbox{ is shorter than }p_i\rangle;$$
$$I_1=I_0+\langle p_i,q_i\mid i\in Q_0, p_i-q_i\in R_1\rangle;$$
$$I_2=I_0+\langle p_i,q_i\mid i\in Q_0, i\notin\mathbb{W}, p_i-q_i\in R_1\rangle+\langle q_i\mid i\in \mathbb{W}, p_i-q_i\in R_1, q_i\mbox{ is shorter than }p_i\rangle;$$
$$I_3=\langle r_i\in \mathbb{P}\mid i\in \mathbb{W}, p_i-q_i\in R_1\rangle=k\mbox{-vector space with basis }\{r_i\in \mathbb{P}\mid i\in \mathbb{W}\}.$$

The following proposition describes when $\mathrm{gr}(A)$ and $A$ are isomorphic.
	
\begin{Prop}[{\cite[Proposition 2.13]{GL}}] \label{isomorphism}
	Let $A=kQ/I$ be a Brauer graph algebra associated with a Brauer graph $G$ and $\mathrm{gr}(A)$ the associated graded algebra of $A$. Then the following statements are equivalent.
	
	\begin{enumerate}[$(1)$]
		\item $A$ is isomorphic to $\mathrm{gr}(A)$ as algebras.
		\item The vertices in the Brauer graph $G$ have the same graded degree.
		\item $\mathbb{W}$ (resp. $\mathbb{P}$) is an empty set.
	\end{enumerate}
\end{Prop}

\section{$\star$-condition, unbalanced edge pair and main result}
	
Throughout this section, we assume that $A=kQ/I$ is a Brauer graph algebra associated with a Brauer graph $G=(V(G),E(G),m)$ and that $\mathrm{gr}(A)=kQ/I'$ its associated graded algebra. Moreover, let $\overline{A}=kQ/I_1$ and $\overline{\mathrm{gr}(A)}=kQ/I_2$ be the associated string algebras in (2.1) and (2.2), respectively. Note that by definition, $A$ and $\overline{A}$ (resp. $\mathrm{gr}(A)$ and $\overline{\mathrm{gr}(A)}$) have the same representation type. In this section, we will define useful notions and state our main results on the infinite representation type of $\mathrm{gr}(A)$.
	
\subsection{$\star$-Condition in a Brauer graph}

\begin{Def}[Compare with {\cite[Definition 3.7]{GL}}] \label{walk-in-graph} Let $u, v$ be two distinct vertices in a Brauer graph $G$.
\begin{enumerate}[$(1)$]
\item A walk from $u$ to $v$ is a sequence $[v_{1},a_{1},v_{2},\ldots, v_{k-1}, a_{k-1},v_{k}]$ of vertices and edges, where $v_{1}=u$, $v_{k}=v$, $a_i$ is an edge incident to the vertices $v_i$ and $v_{i+1}$ for each $1\leq i\leq k-1$, and all vertices (and hence all edges) are pairwise distinct. We often simply write this walk by $[a_{1},\ldots,a_{k-1}]$ and call it walk from edge $a_{1}$ to edge $a_{k-1}$. In particular, when $G$ is a tree, the walk from $u$ to $v$ is unique.
\item The length of a walk from $u$ to $v$ is defined to be the number of edges in this walk and will be denoted by $d_G(u,v)$ whenever the context is clear.
\item We say that a walk $[v_{1},a_{1},v_{2},\ldots, v_{k-1}, a_{k-1},v_{k}]$ is degree decreasing if $\mathrm{grd}(v_{1})\geq \mathrm{grd}(v_{2})\geq\cdots\geq \mathrm{grd}(v_{k})$.
\end{enumerate}
\end{Def}

\begin{Rem} 
		The definition of a walk in Brauer graph is different from the definition of a walk in graph theory, actually, any walk in Brauer graph is identified with a path in graph theory. 
\end{Rem}

\begin{Rem} \label{finite-representation-type}	
According to \cite{GL}, $\mathrm{gr}(A)$ is of finite representation type if and only if $G$ is a Brauer tree with an exceptional vertex $v_0$ of multiplicity $m_0$ such that any walk starting from a specified vertex $v_h$ is degree decreasing, where $v_h$ is defined to be $v_{0}$ when $m_0>1$ or one of the vertices with maximal graded degree when $m_0=1$.
\end{Rem}

In order to generalize our description from finite representation type to infinite representation type, we introduce the following condition on any Brauer graph.

\begin{Def} \label{star-condition} Let $G$ be a Brauer graph and $v_S\defeqi v_L$ an unbalanced edge in $G$. We say that $G$ satisfies $\star$-condition with respect to $v_S\defeqi v_L$ if the following three conditions hold:
\begin{enumerate}[{$(1)$}]		
\item $G_{i,S}\neq G_{i,L}$.
\item $G_{i,S}$ is a tree with $m(v)=1$ for all $v\in V(G_{i,S})$.
\item The unique walk from $v_S$ to any vertex $v_{k}$ in $G_{i,S}$ is degree decreasing.
\end{enumerate}
\end{Def}

\begin{Rem}	
\begin{enumerate}[$(1)$]
	\item $G_{i,L}=G_{i,S}$ for an unbalanced edge $v_S\defeqi v_L$ in a Brauer graph $G$ if and only if $i$ is an edge in some cycle of $G$ if and only if there is another walk from $v_L$ to $v_S$ different from $[i]$. 
    \item By \cite[Theorem 4.5]{GL}, we can formulate the finite representation type using $\star$-condition as follows: $\mathrm{gr}(A)$ is of finite representation type if and only if $G$ is a Brauer tree which satisfies $\star$-condition with respect to any unbalanced edge in $G$. 
\end{enumerate}
\end{Rem}

\begin{Def}[Compare with {\cite[Definition 3.9]{GL}}] \label{simple-string} Let $c_n\ldots c_1$ be a string in $\overline{\mathrm{gr}(A)}$. We say that $c_n\ldots c_1$ is a simple string in $\overline{\mathrm{gr}(A)}$ from $s(c_1)$ to $t(c_n)$ if all $s(c_k)$ are pairwise distinct and $t(c_n)$ is different from $s(c_k)$ for each $1\leq k\leq n$.
\end{Def}
	
\begin{Rem} \label{construct-simple-string}	Similarly as in the proof of \cite[Lemma 3.8]{GL}, for any walk of length $\geq 2$ in $G$ we can get exactly two simple strings in $\overline{\mathrm{gr}(A)}$ corresponding to this walk.
\end{Rem}

We now generalize some results for Brauer tree algebras in \cite{GL} to Brauer graph algebras.

\begin{Lem}[Compare with {\cite[Lemma 3.10]{GL}}] \label{subword-cycle-1} 
	Let $\overline {\mathrm{gr}(A)}=kQ/I_{2}$ and $C=c_n\ldots c_2 c_1$ a string in $\overline{\mathrm{gr}(A)}$, where $s(C)=t(C)=i$. We denote by $u\defeqi v$ the corresponding edge in $G$. If $c_1$ or $c^{-1}_1$ lies in a special cycle at $v$ and $G_{i,v}\neq G_{i,u}$, then $C$ has a substring $C'=c_s\ldots c_2 c_1$ satisfying $s(C')=t(C')=i$ such that the edges  corresponding to vertices $s(c_k)$ ($2\leq k\leq s$) lie in  $G_{i,v}$. Moreover, if $G_{i,v}$ is a tree, then $C'$ has a substring $C_1$ such that $s(C_1)=t(C_1)$ and that $C_1$ or ${C_1}^{-1}$ is a directed string.
\end{Lem}

\begin{proof} The first result is an obvious consequence of $G_{i,v}\neq G_{i,u}$. The  proof  of the second result is identical to \cite[Lemma 3.10]{GL}.
\end{proof}

\begin{Lem}[Compare with {\cite[Lemma 4.3]{GL}}]\label{not-star-condition} 
Let $G$ be a Brauer graph and $\overline {\mathrm{gr}(A)}=kQ/I_{2}$. Suppose that $C=c_n\ldots c_{l}\ldots c_1$ is a string in $\overline{\mathrm{gr}(A)}$ satisfying $l<n$ and that $c_l\ldots c_1$ or $c_1^{-1}\ldots c_l^{-1}$ is an element of $\mathbb{P}$, where $s(c_1)=t(c_n)=i$.  We denote by $v_S\defeqi v_L$ the corresponding unbalanced edge in $G$. If $G_{i,S}\neq G_{i,L}$ and $G_{i,S}$ is a tree, then at least one of the following holds.
\begin{enumerate}[$(1)$]
	\item There is a vertex $v$ with $m(v)\geq 2$ in $G_{i,S}$.
	\item There are some adjacent vertices $v$, $w$ in $G_{i,S}$, such that $d_{G}(v,v_{S})+1=d_{G}(w,v_{S})$ and $\mathrm{grd}(v)<\mathrm{grd}(w)$.
\end{enumerate}	
In other words,  $G$ does not satisfy $\star$-condition with respect to $v_S\defeqi v_L$.	
\end{Lem}	

\begin{proof} Noting that {\cite[Lemma 3.4]{GL}}, {\cite[Lemma 4.1]{GL}} and Lemma \ref{subword-cycle-1}, we have an identical proof to \cite[Lemma 4.3]{GL}.
\end{proof}

The first statement in the following result generalizes \cite[Lemma 3.5]{GL}, and the second one generalizes \cite[Lemma 5.1]{GL} and \cite[Proposition 5.4]{GL}, both are stated in the Brauer tree case in \cite{GL}.
	
\begin{Lem} \label{finite-number} \begin{enumerate}[$(1)$] 
\item Let $E$ be a set consisting of some unbalanced edges in $G$, and for each $i\in E$, let $r_i$ be the element in $\mathbb{P}$ corresponding to the unbalanced edge $i$, where $\mathbb{P}$ is defined in $(2.5)$. If $C$ is a string in $\overline{\mathrm{gr}(A)}$ and $C$ is not a string in $\overline{\mathrm{gr}(A)}/(\bigoplus_{i\in E}r_i)$, then, there exists  $i\in E$ such that  $C$ or $C^{-1}$ has a substring $r_i$. In particular, if $C$ is a string in $\overline{\mathrm{gr}(A)}$ and $C$ is not a string in $\overline{A}$, then $C$ or $C^{-1}$ has a substring lying in the set $\mathbb{P}$.
\item Suppose that $C=c_n\ldots c_{l}\ldots c_1$ is a string in $\overline{\mathrm{gr}(A)}$ satisfying $l<n$ and that $c_l\ldots c_1$ or $c_1^{-1}\ldots c_l^{-1}$ is an element of $\mathbb{P}$, where $s(c_1)=t(c_l)=i$ and $\mathbb{P}$ is defined in $(2.5)$. Denote by $v_S\defeqi v_L$ the corresponding unbalanced edge in $G$. If $G$ satisfies $\star$-condition with respect to $v_S\defeqi v_L$, then 
\begin{enumerate}[$(2.1)$]
	\item $c_n\ldots c_{l+1}$ is a simple substring of $C$ such that $t(c_k)$ is in $G_{i,S}$ for each $l+1\leq k \leq n$ and in particular the string $C$ is not a band in $\overline{\mathrm{gr}(A)}$;
    \item the number of strings $C$ in  $\mathrm{St}(\overline{\mathrm{gr}(A)})$  which contain a substring $r_i$ or $r_i^{-1}$ is equal to $(n_{i}+1)^{2}$, where $n_i$ is the number of edges in $G_{i,S}$.
\end{enumerate}
\item Let $E$ be a set consisting of some unbalanced edges in $G$, and for each $i\in E$, let $r_i$ be the element in $\mathbb{P}$ corresponding to the unbalanced edge $i$. If $G$ satisfies $\star$-condition with respect to any unbalanced edge in $E$,  then  $\overline{\mathrm{gr}(A)}$ and  $\overline{\mathrm{gr}(A)}/(\bigoplus_{i\in E}r_i)$ have the same representation type. In particular, if  $E$ is a set of all unbalanced edges in $G$ and $G$ satisfies $\star$-condition with respect to any unbalanced edge, then, by the relationship $\overline{A}\cong\overline{\mathrm{gr}(A)}/(\bigoplus_{i\in E}r_i)$, $\overline{\mathrm{gr}(A)}$ and $\overline{A}$ have the same representation type.
\end{enumerate}
\end{Lem}

\begin{proof}
(1) Since the $k$-vector space $\bigoplus_{i\in E}r_i$ is an ideal of $\overline{\mathrm{gr}(A)}$, the conclusion is clear.

(2.1) To show that the edge corresponding to $t(c_k)$ is in $G_{i,S}$ for $l+1\leq k\leq n$, it suffices to prove that $t(c_k)\neq i$ for $l+1\leq k\leq n$. Suppose on the contrary that there exists $l+1\leq m \leq n$ such that $t(c_m)=i$ and $t(c_k)\neq i$ for $l+1\leq k\leq m-1$. Since $G$ satisfies $\star$-condition with respect to $v_S\defeqi v_L$, we have that $G_{i,S}\neq G_{i,L}$ and $G_{i,S}$ is a tree. For the substring $c_m\ldots c_{l}\ldots c_1$, where $s(c_1)=t(c_m)=i$, the string $c_m\ldots c_{l}\ldots c_1$ satisfies the conditions of Lemma \ref{not-star-condition}, which then leads to a contradiction.

Next we show that $c_n\ldots c_{l+1}$ is a simple string. It suffices to show that all $t(c_k)$ are pairwise distinct for $l\leq k\leq n$. Suppose that there exist $k$ and $t$ satisfying $l\leq t<k\leq n$ such that $t(c_k)=t(c_t)=s(c_{t+1})$ and that $t(c_m)$ is different from $t(c_s)$ for each $l\leq m<k$ and $l\leq s<m$. Repeating the similar proof as above, we still get a contradiction.

(2.2) Since $G$ satisfies $\star$-condition with respect to $v_S\defeqi v_L$, this proof is  identical to the proof of  \cite[Proposition 5.4]{GL}.

(3) Since $G$ satisfies $\star$-condition with respect to any unbalanced edge in $E$, (2.1) shows that the band modules over the two algebras $\overline{\mathrm{gr}(A)}$ and  $\overline{\mathrm{gr}(A)}/(\bigoplus_{i\in E}r_i)$ are the same, and (2.2) shows that the number of string $\overline{\mathrm{gr}(A)}$-modules is equal to the number of string $\overline{\mathrm{gr}(A)}/(\bigoplus_{i\in E}r_i)$-modules plus $\sum_{i\in E}(n_{i}+1)^{2}$, it follows that $\overline{\mathrm{gr}(A)}$ and  $\overline{\mathrm{gr}(A)}/(\bigoplus_{i\in E}r_i)$ have the same representation type.
\end{proof}

\begin{Prop}\label{1-domestic-tree} Let $G$ be a Brauer graph which is a tree with $m(v)=2$ for exactly two vertices $v=w_0,w_1$ in $V(G)$ and $m(v)=1$ for all $v\neq w_0,w_1$. Then the following two conditions are equivalent:
\begin{enumerate}[$(1)$]
	\item $G$ satisfies $\star$-condition with respect to any unbalanced edge $v_S\defeqi v_L$;
    \item $w_0$ and $w_1$ are in $G_{i,L}$ for any unbalanced edge $v_{S}\defeqi v_{L}$ in $G$.
\end{enumerate}
Moreover, if $G$ satisfies $\star$-condition with respect to any unbalanced edge, then $\mathrm{gr}(A)$ is domestic. In particular, $\mathrm{gr}(A)$ is $1$-domestic.
\end{Prop}
	
\begin{proof} Assume that $w_0$ and $w_1$ are in $G_{i,L}$ for any unbalanced edge $v_{S}\defeqi v_{L}$ in $G$. Since $G$ is a tree with $m(v)=2$ for exactly $v=w_0,w_1$ and $w_0$ and $w_1$ are in $G_{i,L}$ for $v_{S}\defeqi v_{L}$, $G_{i,S}\neq G_{i,L}$ and $G_{i,S}$ is a tree with $m\equiv1$. Then the conditions $(1)$ and $(2)$ of $\star$-condition hold. We next show that the condition $(3)$ of $\star$-condition holds.
		
We suppose, on the contrary that, there exists a vertex $w$ in $G_{i,S}$ for some unbalanced edge $v_S\defeqi v_L$ such that the walk $[v_{1},a_{1},v_{2},\ldots,$ $v_{k-1},a_{k-1},v_{k}]$ from $v_S$ to $w$ is not degree decreasing, where $v_1=v_{S}$ and $v_k=w$. In other words, there exists an unbalanced edge $v_{j}\defeqtj v_{j+1}$ with $\mathrm{grd}(v_{j})<\mathrm{grd}(v_{j+1})$ for some $1\leq j\leq k-1$. Since $d_G(v_j,v_L)+1=d_G(v_{j+1},v_L)$ and $w_0$ and $w_1$ are in $G_{i,L}$, $G_{i,L}\subseteq G_{a_j,S}$ and $w_0$ and $w_1$ are in $G_{a_j,S}$. It contradicts the condition that $w_0$ and $w_1$ are in $G_{a_j,L}$.

Conversely, assume that $G$ satisfies $\star$-condition with respect to any unbalanced edge $v_S\defeqi v_L$. Suppose that there is some unbalanced edge $v_S\defeqi v_L$ such that $w_0$ or $w_1$ is in $G_{i,S}$. It clearly contradicts the condition $(2)$ of $\star$-condition.

Now assume that $G$ satisfies $\star$-condition with respect to any unbalanced edge. Then,  by Lemma \ref{finite-number}, $\overline{\mathrm{gr}(A)}$ and $\overline{A}$ have the same representation type. It follows from Theorem \ref{Brauer-graph-domestic} that $\overline{\mathrm{gr}(A)}$ is $1$-domestic. Hence $\mathrm{gr}(A)$ is $1$-domestic.
\end{proof}
	
\begin{Prop}\label{domestic-cycle}
		Let $G$ be a Brauer graph with a unique cycle and $m(v)=1$ for all $v\in V(G)$. Then the following two conditions are equivalent:
\begin{enumerate}[$(1)$]
	\item $G$ satisfies $\star$-condition with respect to any unbalanced edge $v_S\defeqi v_L$;
    \item all edges in the unique cycle are not unbalanced edges and the unique cycle is in $G_{i,L}$ for any unbalanced edge $v_{S}\defeqi v_{L}$ in $G$.
\end{enumerate}
Moreover, if $G$ satisfies $\star$-condition with respect to any unbalanced edge, then $\mathrm{gr}(A)$ is domestic. In particular, if the unique cycle is of odd length (resp. even length), then $\mathrm{gr}(A)$ is $1$-domestic (resp. $2$-domestic).
\end{Prop}
	
\begin{proof}  Assume that all edges in the unique cycle are not unbalanced edges and the unique cycle is in $G_{i,L}$ for any unbalanced edge $v_{S}\defeqi v_{L}$ in $G$. Since $m(v)=1$ for all $v\in V(G)$,  $m(v)=1$ for all $v\in V(G_{i,S})$ for $v_S\defeqi v_L$. Since all edges in the unique cycle are not unbalanced edges and the unique cycle is in $G_{i,L}$ for $v_S\defeqi v_L$, $G_{i,S}\neq G_{i,L}$ and $G_{i,S}$ is a tree. Then the conditions $(1)$ and $(2)$ of $\star$-condition hold. The condition $(3)$ of $\star$-condition holds by using a similar approach to the proof of Proposition \ref{1-domestic-tree}.

Conversely, assume that $G$ satisfies $\star$-condition with respect to any unbalanced edge $v_S\defeqi v_L$. If there is some edge in the unique cycle is an unbalanced edge, then it contradicts the condition $(1)$ of $\star$-condition. If there is some unbalanced edge $v_S\defeqi v_L$ such that the unique cycle is in $G_{i,S}$, then it contradicts the condition $(2)$ of $\star$-condition.

Now assume that $G$ satisfies $\star$-condition with respect to any unbalanced edge. Then, by Lemma \ref{finite-number}, $\overline{\mathrm{gr}(A)}$ and $\overline{A}$ have the same representation type. It follows from Theorem \ref{Brauer-graph-domestic} that if the unique cycle in $G$ is of odd length (resp. even length), then $\mathrm{gr}(A)$ is $1$-domestic (resp. $2$-domestic).
\end{proof}

The next result deals with a case where the cardinality of $\mathrm{Ba}(\overline{\mathrm{gr}(A)})$ is infinite. 

\begin{Lem}\label{nondomestic} If there are two distinct bands $b_1=c_m\ldots c_{l+1} c_l\ldots c_1$ and $b_2=c'_{m'}\ldots c'_{l+1} c_l\ldots c_1$ in $\overline{\mathrm{gr}(A)}$, where $s(c_1)=t(c_l)$, $c_l\ldots c_1$ is a directed substring, $c_{l+1}=c'_{l+1}$ is an inverse arrow, and $s(c_i)\neq s(c_1)$ (resp. $s(c'_i)\neq s(c_1)$) for $l+1<i\leq m$ (resp. $l+1<i\leq m'$). Then the cardinality of $\mathrm{Ba}(\overline{\mathrm{gr}(A)})$ is infinite and $\mathrm{gr}(A)$ is not of polynomial growth.
\end{Lem}

\begin{proof} From the assumption of the two bands $b_1$ and $b_2$, we have that all powers of $b_2b_1$ are strings in $\overline{\mathrm{gr}(A)}$.
Moreover, since $b_1$ and $b_2$ are distinct and $s(c_i)\neq s(c_1)$ (resp. $s(c'_i)\neq s(c_1)$) for $l+1<i\leq m$ (resp. $l+1< i\leq m'$), $b_2b_1$ is not a power of a string of smaller length. Then $b_2b_1$ is also a band. Similarly, for any positive integer $k$, $b^k_2b_1$ is a band. Then the cardinality of $\mathrm{Ba}(\overline{\mathrm{gr}(A)})$ is infinite.

In order to prove that $\mathrm{gr}(A)$ is not of polynomial growth.  We just prove that $\overline{\mathrm{gr}(A)}$ is not of polynomial growth. Let $n'_i$ be the length of $b_i$ for $i=1,2$. When $n'_1=n'_2$, there are pairwise distinct elements $b_2^{k_n-1}b_1^{k_n}b_2^{k_{n-1}+1}b_1^{k_{n-1}}b_2^{k_{n-2}}b_1^{k_{n-2}}\ldots b_2^{k_1}b_1^{k_1}$  in $\mathrm{Ba}(\overline{\mathrm{gr}(A)})$, where $n$ and $k_n$ are positive integers greater than $1$, and $k_i$ is positive integer such that  $\sum_{i=1}^{n}k_i=mn'_2$ for each $1\leqslant i\leqslant n-1$.
Considering the corresponding indecomposable band modules with dimension $d$, where $d=2mn'_1n'_2$. We have that the number of the above indecomposable band modules is $\sum_{n=2}^{mn'_2-1}\binom{mn_2'-2}{n-1}=2^{mn'_2-2}-1$. Therefore $\mu_{\overline{\mathrm{gr}(A)}}(d)\geq 2^{mn'_2-2}-1$ and there is no positive integer $m$ such that $\mu_{\overline{\mathrm{gr}(A)}}(d)\leq d^{m}$ for all $d\geq 2$. Hence $\mathrm{gr}(A)$ is not of polynomial growth.

When $n'_1\neq n'_2$, without loss of generality, we can assume  $n'_2<n'_1$. There are pairwise distinct elements $b_2^{k_n+m(n'_1-n'_2)}b_1^{k_n}b_2^{k_{n-1}}b_1^{k_{n-1}}\ldots b_2^{k_1}b_1^{k_1}$  in $\mathrm{Ba}(\overline{\mathrm{gr}(A)})$, where $n$ and $k_i$ are positive integers with $\sum_{i=1}^{n}k_i=mn'_2$ for each $1\leqslant i\leqslant n$.  Considering the corresponding indecomposable band modules with dimension $d$, where $d=2mn'_1n'_2$.  
we have that the number of the above indecomposable band modules is $\sum_{n=1}^{mn'_2}\binom{mn_2'-1}{n-1}=2^{mn'_2-1}$. Therefore $\mu_{\overline{\mathrm{gr}(A)}}(d)\geq 2^{mn'_2-1}$ and there is no positive integer $m$ such that $\mu_{\overline{\mathrm{gr}(A)}}(d)\leq d^{m}$ for all $d\geq 2$. Hence $\mathrm{gr}(A)$ is not of polynomial growth.
\end{proof}

\subsection{Unbalanced edge pair in a Brauer tree} In order to describe the domestic $\overline{\mathrm{gr}(A)}$ when $G$ is a Brauer tree, we introduce the notion of unbalanced edge pair.

\begin{Def} \label{edge-pair}
Let $G$ be a Brauer tree with an exceptional vertex $v_0$ of multiplicity  $m_0$.
\begin{enumerate}[$(1)$]
	\item We call $(i,j)$ an unbalanced edge pair if $j$ is an edge in $G_{i,S}$ and $d_G(v^{(j)}_S,v^{(i)}_S)+1=d_G(v^{(j)}_L,v^{(i)}_S)$, where $v^{(i)}_S\defeqi v^{(i)}_L$ and $v^{(j)}_S\defeqj v^{(j)}_L$ are two unbalanced edges in $G$.
	\item We define $\kappa_0$ to be the number of unbalanced edges $v_S \defeqi v_L$ in $G$ such that the exceptional vertex $v_0$ is a vertex in $G_{i,S}$.
	\item We define $\kappa_1$ to be the number of unbalanced edge pairs in $G$.
\end{enumerate}
\end{Def}

\begin{Rem}	
\begin{enumerate}[$(1)$]	
	\item $(i,j)$ is an unbalanced edge pair if and only if $(j,i)$ is an unbalanced edge pair.
    \item $\kappa_1=0$ if and only if the unique walk from $v_S$ to any vertex in $G_{i,S}$ is degree decreasing for any unbalanced edge $v_S\defeqi v_L$ in $G$.
	\item By \cite[Theorem 4.5]{GL}, $\mathrm{gr}(A)$ is of finite representation type if and only if $G$ is a Brauer tree such that $\kappa_0(m_0-1)+\kappa_1=0$.	
\end{enumerate}
\end{Rem}
	
We have the following observations about $\kappa_0$ and $\kappa_1$.
	
\begin{Lem}\label{e_1}
 If $\kappa_1\neq 0$, then there are two unbalanced edges $v^{(i)}_S\defeqi v^{(i)}_L$ and  $v^{(j)}_S\defeqj v^{(j)}_L$ in $G$ such that the exceptional vertex $v_0$ is a vertex in $G_{i,S}$ and $(i,j)$ is an unbalanced edge pair. In particular, if $\kappa_1\neq 0$, then $\kappa_0\neq 0$.	
\end{Lem}
	
\begin{proof} Since $\kappa_1\neq0$, there are two unbalanced edges $v^{(i)}_S\defeqi v^{(i)}_L$ and  $v^{(j)}_S\defeqj v^{(j)}_L$ in $G$ such that $(i,j)$ is an unbalanced edge pair. Without loss of generality, we assume that $v_0$ is a vertex in $G_{i,L}$. Since  $G\setminus i=G_{i,S}\bigcup G_{i,L}$ and $(i,j)$ is an unbalanced edge pair, $G_{i,L}\subseteq G_{j,S}$ and $v_0$ is a vertex in $G_{j,S}$. We get our desired result.
\end{proof}

\begin{Lem}\label{e_1-character}
Let $G$ be a Brauer tree with an exceptional vertex $v_0$ of multiplicity  $m_0$. Then $\kappa_1\geq 2$ if and only if there are three unbalanced edges $v^{(i)}_S\defeqi v^{(i)}_L$, $v^{(j)}_S\defeqj v^{(j)}_L$ and  $v^{(k)}_S\defeqk v^{(k)}_L$ in $G$ such that $(i,j)$ and $(i,k)$ are unbalanced edge pairs. 	
\end{Lem}
	
\begin{proof}
		``$\Longleftarrow$''
		It is obvious to get $\kappa_1\geq 2$ if $(i,j)$ and $(i,k)$ are unbalanced edge pairs. 	
		
		``$\Longrightarrow$''
		Since $\kappa_1\geq 2$, there are at least two unbalanced edge pairs. Without loss of generality, we assume that $(i,j)$ and $(k,l)$ are two unbalanced edge pairs, where the unbalanced edges $v^{(i)}_S\defeqi v^{(i)}_L$, $v^{(j)}_S\defeqj v^{(j)}_L$,   $v^{(k)}_S\defeqk v^{(k)}_L$ and $v^{(l)}_S \defeqL v^{(l)}_L$ are pairwise distinct and $k$ is an edge in $G_{i,S}$. There are two cases to be considered.
		
		{\it Case 1.} If $d_G(v^{(k)}_S,v^{(i)}_S)+1=d_G(v^{(k)}_L,v^{(i)}_S)$, then $(i,k)$ is an unbalanced edge pair. Moreover, $(i,j)$ is also an unbalanced edge pair. We have that $(i,j)$ and $(i,k)$ are two unbalanced edge pairs.
		
		{\it Case 2.} If $d_G(v^{(k)}_S,v^{(i)}_S)-1=d_G(v^{(k)}_L,v^{(i)}_S)$, since the unbalanced edge $l$ is an edge in $G_{k,S}$ with $d_G(v^{(l)}_S,v^{(k)}_S)+1=d_G(v^{(l)}_L,v^{(k)}_S)$ and $G_{k,S}\subseteq G_{i,S}$, then $l$ is an edge in $G_{i,S}$ with $d_G(v^{(l)}_S,v^{(i)}_S)+1=d_G(v^{(l)}_L,v^{(i)}_S)$ and therefore $(i,l)$ is an unbalanced edge pair. Moreover, $(i,j)$ is also an unbalanced edge pair. We have that $(i,j)$ and $(i,l)$ are two unbalanced edge pairs.
\end{proof}

\subsection{Main result and consequences} We now state our main result, which characterizes when $\mathrm{gr}(A)$ is domestic.

\begin{Thm}\label{main-result}
Let $A$ be the Brauer graph algebra associated with a Brauer graph $G=(V(G),E(G))$ and $\mathrm{gr}(A)$ the graded algebra associated with the radical filtration of $A$, where $V(G)$ is the vertex set and $E(G)$ is the edge set. Let $\kappa_0$ and $\kappa_1$ be defined as in Definition \ref{edge-pair}. Then 
the following three statements are equivalent.
\begin{enumerate}[$(a)$]
	\item $\mathrm{gr}(A)$ is of polynomial growth.
	\item $\mathrm{gr}(A)$ is domestic.
	\item The cardinality of $\mathrm{Ba}(\overline{\mathrm{gr}(A)})$ is finite, where  $\overline{\mathrm{gr}(A)}$ is defined in $(2.2)$.
\end{enumerate}
Furthermore, for domestic type, we have the following.
\begin{enumerate}
	\item[$(b1)$] $\mathrm{gr}(A)$ is $1$-domestic if and only if one of the following holds.
	\begin{enumerate}[$(1)$]
		\item $G$ is a Brauer tree with an exceptional vertex $v_0$ of multiplicity $m_0$ such that $\kappa_0(m_0-1)+\kappa_1=1$.
		\item $G$ is a tree and there exist two distinct vertices $w_0,w_1$, such that the following conditions hold:
       \begin{enumerate}[$(2.1)$]
			\item $m(w_0)=m(w_1)=2$ and $m(v)=1$ for $v\neq w_0,w_1$.
			\item $\mathrm{grd}(w_0)=\mathrm{grd}(w_1)$.
			\item Any walk from $w_0$ (or from $w_1$) is degree decreasing.
		\end{enumerate}
		\item $G$ is a graph with a unique cycle of odd length and $m(v)=1$ for all $v\in V(G)$, and satisfies the following conditions hold:
		\begin{enumerate}[$(3.1)$]
			\item $\mathrm{grd}(u)=\mathrm{grd}(v)$ for any two vertices $u$ and $v$ in the unique cycle.
			\item  Any walk from any vertex in the unique cycle is degree decreasing.
		\end{enumerate}
	\end{enumerate}
	\item[$(b2)$] $\mathrm{gr}(A)$ is $2$-domestic if and only if $G$ satisfies the following conditions.
	\begin{enumerate}[$(1)$]
		\item $G$ is a graph with a unique cycle of even length and $m(v)=1$ for all $v\in V(G)$.
		\item $\mathrm{grd}(u)=\mathrm{grd}(v)$ for any two vertices $u$ and $v$ in the unique cycle.
		\item Any walk from any vertex in the unique cycle is degree decreasing.
	\end{enumerate}
	\item[$(b3)$] $\mathrm{gr}(A)$ is not $n$-domestic for $n\geq3$.
\end{enumerate}
\end{Thm}	

\begin{proof} Since $A$ and $\overline{A}$ (resp. $\mathrm{gr}(A)$ and $\overline{\mathrm{gr}(A)}$) have the same representation type, and since $\overline{A}$ is a quotient of the algebra $\overline{\mathrm{gr}(A)}$, if the Brauer graph algebra $A$ is nondomestic (resp. $A$ is not of polynomial growth), then $\mathrm{gr}(A)$ is nondomestic (resp. $\mathrm{gr}(A)$ is not of polynomial growth). By Theorem \ref{Brauer-graph-domestic}, in order to describe when $\mathrm{gr}(A)$ is domestic or is of polynomial growth, it suffices to study $\overline{\mathrm{gr}(A)}$ in the case when $G$ is a Brauer tree or in the cases (a) and (b) in Theorem \ref{Brauer-graph-domestic}. The descriptions in these cases are given in Proposition \ref{Brauer-tree-domestic}, Proposition \ref{tree-another} and Proposition \ref{graph-another}, respectively.
\end{proof}

The main result gives the following consequences.

\begin{Cor} Let $A$ be a Brauer tree algebra and $\mathrm{gr}(A)$ the associated graded algebra of $A$. If $\kappa_1>1$ or $\kappa_1=1$ and $m_0>1$, then $\mathrm{gr}(A)$ is not of polynomial growth.
\end{Cor}

\begin{Cor}
Let $A$ be a Brauer graph algebra and $\mathrm{gr}(A)$ the associated graded algebra of $A$. Then $\mathrm{gr}(A)$ is $n$-domestic if and only if the cardinality of $\mathrm{Ba}(\overline{\mathrm{gr}(A)})$ is $n$. \end{Cor}

According to Theorem \ref{Brauer-graph-domestic}, the similar result as in the above corollary holds for any Brauer graph algebra $A$, that is, $A$ is $n$-domestic if and only if the cardinality of $\mathrm{Ba}(\overline{A})$ is $n$. It would be interesting to know whether there is a similar result for any domestic special biserial algebra.
 
Comparing with the results in \cite{GL}, we would also like to give the following remarks on the relationships between the Auslander-Reiten quivers of $\mathrm{gr}(A)$ and $A$.

\begin{Rem}
\begin{enumerate}[$(1)$]
\item Similar as the discussion in \cite[Section 5]{GL}, for domestic $\mathrm{gr}(A)$ except the Brauer tree case, based on Lemma \ref{finite-number}, Proposition \ref{1-domestic-tree}, Proposition \ref{domestic-cycle}, we can prove that, the Auslander-Reiten quiver of $\overline{A}$ is obtained from the Auslander-Reiten quiver of $\overline{\mathrm{gr}(A)}$ by removing several diamonds.
	\item When $G$ is a Brauer tree and $\mathrm{gr}(A)$ is domestic, the situation is more complicated. In this case we have the following conjecture on the Auslander-Reiten quiver $\Gamma$ of $\overline{\mathrm{gr}(A)}$:
\begin{enumerate}[$(2.1)$]
\item  $\Gamma$ consists of components of the form $\mathbb{Z}\tilde{A}_{p,q}$ and components of the form $\mathbb{Z}A_{\infty}/\langle\tau^n\rangle$ (both components are up to deleting some diamonds). Moreover, when $m_0=1$, $\Gamma$ has a component $\mathbb{Z}\tilde{A}_{p,q}$ satisfying $p+q=n_i+n_j+2$ with $(i,j)$ the unique unbalanced edge pair in $G$; when $m_0=2$, $\Gamma$ has a component $\mathbb{Z}\tilde{A}_{n_i+1,|E(G)|}$ with $i$ the unique unbalanced edge in $G$ such that the exceptional vertex $v_0$ is in $G_{i,S}$.
\end{enumerate} 
\item  We note that in the picture of \cite[Remark 5.14]{GL}, the obtained  part $W$ in the  Auslander-Reiten quiver  of $\overline{\mathrm{gr}(A)}$ may be different from the beginning wing $W$, since the obtained  part may contain new inserted diamonds.
\end{enumerate}		
\end{Rem}

\section{The case that $G$ is a Brauer tree} In this section, we describe when $\mathrm{gr}(A)=kQ/I'$ is domestic in the case when $G=(V(G),E(G),m)$ is a Brauer tree with an exceptional vertex $v_0$ of multiplicity $m_0$, where $V(G)$ is the vertex set, $E(G)$ is the edge set and $m$ is the multiplicity function of $G$. Let  $\kappa_0$ and  $\kappa_1$ be defined in Definition \ref{edge-pair}.

Since the number of string modules of a given dimension is finite, it suffices to consider band modules when we consider the representation type of a representation-infinite string algebra. The following lemma is useful when we consider two related representation-infinite string algebras.
	
\begin{Lem}\label{same-number}
Let $\Lambda=kQ/I$ and $\Gamma=\Lambda/J$ be two representation-infinite string algebras, where $J$ is an ideal of $\Lambda$ with $\rad^m(\Lambda)\subseteq J\subseteq \rad^2(\Lambda)$ for some $m\geq 2$. Suppose that for any indecomposable $\Lambda$-module $M$ satisfying $JM\neq 0$, $M$ is a string $\Lambda$-module. Then $\Gamma$ is of polynomial growth (resp. domestic) if and only if $\Lambda$ is of polynomial growth (resp. domestic).
\end{Lem}
	
\begin{proof} Since the algebra $\Gamma$ is a quotient of the algebra $\Lambda$, we have that any band $\Gamma$-module can be considered as a band $\Lambda$-module. Moreover, from the assumption that $M$ is a string $\Lambda$-module for any indecomposable $\Lambda$-module $M$ satisfying $JM\neq 0$, it follows that any band $\Lambda$-module is also a band $\Gamma$-module. Hence there is a one to one correspondence between band $\Lambda$-modules and band $\Gamma$-modules. Combining the remark before this lemma, we get the desired result.
\end{proof}

\begin{Rem}
We have used a special case of the above lemma in Lemma \ref{finite-number} (3), where $\Lambda=\overline{\mathrm{gr}(A)}$ and $\Gamma=\overline{A}$.
\end{Rem}

\begin{Prop}\label{Brauer-tree-not-polynomial}
		Let $\overline{\mathrm{gr}(A)}=kQ/I_2$ be defined in $(2.2)$ and $G$ the associated Brauer tree with an exceptional vertex $v_0$ of multiplicity  $m_0$. If $m_0\geq3$ and $\kappa_0\neq 0$. Then $\overline{\mathrm{gr}(A)}$ and $\mathrm{gr}(A)$ are not of polynomial growth.
\end{Prop}
	
\begin{proof}	
We only need to prove that $\overline{\mathrm{gr}(A)}$ is not of polynomial growth. 
Since $\kappa_0\neq 0$, we have that there is an unbalanced edge $v_S\defeqi v_L$ such that the exceptional vertex $v_0$ is in $G_{i,S}$. Let $i_1<i_2<\cdots<i_n<i_1$ be the cyclic ordering at $v_L$, where $i_1=i$ and $n=\mathrm{grd}(v_L)=\mathrm{val}(v_L)$.  Note that $n>2$. Let $E_1=\{i_1,i_2,\ldots,i_n\}$ and $E_2=E(G_{i,S})\cup E_1$, where the edge $i_k$ is incident to $v_L$ and $v'_k$ for any $2\leq k\leq n$, and $E(G_{i,S})$ is the edge set of $G_{i,S}$. We denote by $E_3$ the set of all unbalanced edges in $G_{i,S}$ and by $r_j$ the element in $\mathbb{P}$ corresponding to the unbalanced edge $j$ in $E_3$. We have algebra homomorphisms as follows.
$$\mathrm{gr}(A)\twoheadrightarrow\overline {\mathrm{gr}(A)}\twoheadrightarrow \overline {\mathrm{gr}(A)}/(\overline {\mathrm{gr}(A)}e\overline {\mathrm{gr}(A)}\oplus \bigoplus\limits_{j\in E_3} r_j),$$ 
where  $e_i$ is the primitive idempotent in $Q$ corresponding to the edge $i$ in $E(G)\setminus E_2$ and $e=\sum_{i\in E(G)\setminus E_2}e_i$.
		
Let $C$ be the above algebra $\overline {\mathrm{gr}(A)}/(\overline {\mathrm{gr}(A)}e\overline {\mathrm{gr}(A)}\oplus \bigoplus_{j\in E_3} r_j)$. Then $C=kQ'/I_{C}$, where $I_{C}$ is an admissible ideal in $kQ'$ and $Q'$ is a subquiver of $Q$ by removing all vertices corresponding to the edges in $E(G)\setminus E_2$ and all related arrows. Note that $C$ is a string algebra and is representation-infinite.

We next construct a related Brauer graph $G'$. Let $G'=(V(G_{i,S})\cup\{v_L,v'_2,v'_3,\ldots,v'_n\},E_2,m)$, where  $V(G_{i,S})$ is the vertex set of $G_{i,S}$, $m(v_0)=m_0$, $m(v_L)=2$ and $m(v)=1$ for the other vertices $v$. We may visualise the underlying graph of $G'$ from $G$ as follows: 

	$$
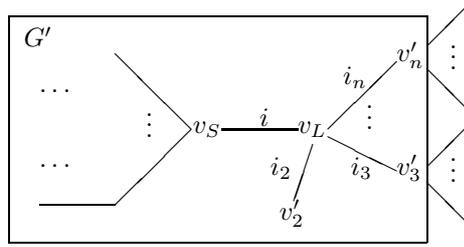
\begin{picture}(35.00,63.00)  
  \unitlength=1.00mm \special{em:linewidth 0.4pt}
  \linethickness{0.4pt}
  \put(0,12){\line(1,0){10}} \put(10,11.5){$v_L$} \put(14,12){\line(1,1){9}}  \put(27,23){\line(1,1){6}} \put(19,12){$\vdots$}    \put(12,10){\line(-1,-3){2.5}}\put(6.4,6){$i_2$} \put(7.6,0){$v'_2$}     \put(14,11  ){\line(2,-1){9}}  \put(17,6){$i_3$} \put(23,6){$v'_3$}    \put(16,18){$i_n$}  \put(23,21){$v'_n$} \put(30,20){$\vdots$}    \put(27,20){\line(1,-1){6}}
  \put(27,7){\line(1,1){6}}  
  \put(30,5){$\vdots$}
  \put(27,5){\line(1,-1){6}}
  \put(5,12.5){$i$} 
  \put(-3.8,11.5){$v_S$}\put(-4,12){\line(-1,1){10}}   \put(-4,12){\line(-1,-1){10}}\put(-10,11){$\vdots$}   
  \put(-24,2){\line(1,0){10}}\put(-24,7){$\ldots$}   \put(-24,17){$\ldots$}   
  \put(-28,-3){\line(1,0){55}} \put(27,-3){\line(0,1){30}}
  \put(-28,27){\line(1,0){55}}  \put(-28,-3){\line(0,1){30}}
  \put(-26,23){$G'$}
 \end{picture}$$
 
Let $B$ be the Brauer graph algebra associated with the new Brauer graph $G'$ and $\overline{B}=kQ'/I_{B,1}$ the corresponding string algebra defined in $(2.1)$. By the construction of the Brauer graph $G'$ and Theorem \ref{Brauer-graph-domestic}, we have that $B$ is not of polynomial growth. Therefore, $\overline{B}$ is not of polynomial growth.

Now let $E_4=\{i_k|\mathrm{grd}(v_L)\leq \mathrm{grd}(v_k'), 1<\mathrm{val}(v_k'), 2\leq k\leq n\}$ (which is also defined in the original Brauer graph $G$), where the edge $i_k$ is incident to $v_L$ and $v'_k$ in $G$ for any $2\leq k\leq n$. We can get the following algebra isomorphism from their constructions
		
		$$C\cong \overline{B}/(\bigoplus\limits_{i_k\in E_1\setminus E_4} \rad^{n+1}(P_{i_k})\oplus\bigoplus\limits_{i_k\in E_4} \rad^{n}(P_{i_k})),$$
		\\ where $P_{i_k}$ is the projective cover of the simple $\overline{B}$-module $S_{i_k}$ corresponding to the edge $i_k$ in $G'$ for each $1\leq k\leq n$ and where $J:=\bigoplus_{i_k\in E_1\setminus E_4} \rad^{n+1}(P_{i_k})\oplus\bigoplus_{i_k\in E_4} \rad^{n}(P_{i_k})$ is an ideal of $\overline{B}$. Clearly $J\subseteq \rad^2(\overline{B})$.

Note that there are two arrows starting and ending at vertex $i$ of $Q'$ and there is one arrow starting and ending at vertex $i_k$ of $Q'$ for all $2\leq k\leq n$. So $Q'$ contains the following subquiver:	
$$
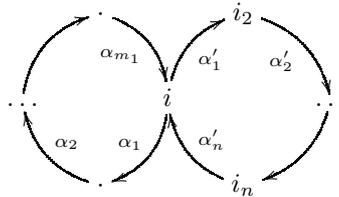
\xymatrix@C-0.8em@R-0.5em@r{
			&\cdot \ar@/^0.63pc/[rd]_{\alpha_{m_1}} &                                  &  i_2\ar@/^0.635pc/[rd]_{\alpha_{2}'}    &  \\
			\ldots\ar@/^0.635pc/[ru]    &       &   i\ar@/^0.63pc/[ru]_{\alpha_{1}'}\ar@/^0.63pc/[ld]_{\alpha_{1}} &    &\ldots\ar@/^0.63pc/[ld]\\
			&           \cdot\ar@/^0.63pc/[lu]_{\alpha_{2}}     &        & i_n \ar@/^0.63pc/[lu]_{\alpha_{n}'} & \quad,
}$$
where $m_1=\mathrm{val}(v_S)$. Moreover, $J$ can be generated by $\alpha'_{k-1}\ldots\alpha'_1\alpha'_n\ldots\alpha'_{k+1}\alpha'_k$ ($i_k\in E_4$) and $\alpha'_k\ldots\alpha'_1\alpha'_n\ldots\alpha'_{k+1}\\\alpha'_k$ ($i_k\in E_1\setminus E_4$). For any band in $\overline{B}$, we have the following claim.
		
{\it Claim:} For any band $b$ in $\overline{B}$, $b$ does not have the substring $\alpha'_{k-1}\ldots\alpha'_1\alpha'_n\ldots\alpha'_{k+1}\alpha'_k$ for any $i_k$ in $E_4$, and $b$ does not have the substring $\alpha'_{k}\ldots\alpha'_1\alpha'_n\ldots\alpha'_{k+1}\alpha'_k$ for any $i_k$ in $E_1\setminus E_4$  (possibly after rotation or taking inverse of $b$). That is, any band in $\overline{B}$ gives in fact a band in the quotient algebra $C$.

If the above claim is true, then, the ideal $J$ satisfies the condition in Lemma \ref{same-number}, and therefore $C$ is not of polynomial growth. It follows that $\overline{\mathrm{gr}(A)}$ is not of polynomial growth. This is our desired result.
		
{\it Proof of Claim.} Suppose on the contrary that $b$ has the substring $\alpha'_{k-1}\ldots\alpha'_1\alpha'_n\ldots\alpha'_{k+1}\alpha'_k$ for some $i_k\neq i_1$ and $b=c_s\ldots c_1\alpha'_{k-1}\ldots\alpha'_1\alpha'_n\ldots\alpha'_{k+1}\alpha'_k$. Since $s(b)=t(b)$ and there is only one arrow starting and ending at vertex $i_k$ for $2\leq k\leq n$, $s(\alpha'_k)=t(c_s)=i_k$ and $c_s=\alpha'_{k-1}$. We rotate $b$ to $c_{s-1}\ldots c_1\alpha'_{k-1}\ldots\alpha'_1\alpha'_n\ldots\alpha'_{k}\alpha'_{k-1}$. If $i_{k-1}\neq i_1$, then we can repeat the above step and therefore we can assume that $b$ has the substring $\alpha'_1\alpha'_n\ldots\alpha'_2\alpha'_1$. We may assume that $b=c_s\ldots c_1\alpha'_1\alpha'_n\ldots\alpha'_2\alpha'_1$. Since $s(\alpha'_1)=t(c_s)$ and there is only one arrow starting and ending at vertex $i_k$ for $2\leq k\leq n$, we have that $b$ has the substring $\alpha'_n\ldots\alpha'_2\alpha'_1\alpha'_n\ldots\alpha'_2\alpha'_1$. It contradicts the fact  that $\alpha'_n\ldots\alpha'_2\alpha'_1\alpha'_n\ldots\alpha'_2\alpha'_1$ is in the ideal $I_{B,1}$ (indeed it is an element of $\soc(B)$). This finishes the proof of our claim.
\end{proof}

We give an example to illustrate the above result.

\begin{Ex}
Let $G$ be the following Brauer tree with $m_0=3$.		
$$
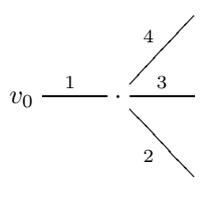
\xymatrix{
			& & \cdot	\\
			v_0\ar@{-}[r]^{1}& \cdot\ar@{-}[ur]^{4}\ar@{-}[r]^{3}\ar@{-}[dr]_{2}& \cdot\\
			& & \cdot}$$
		
Let $A=kQ/I$ be the Brauer tree algebra associated with $G$ and $\mathrm{gr}(A)$ the associated graded algebra of $A$. The quiver $Q$ is as follows.
$$
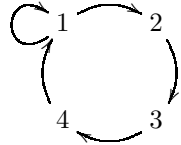
\xymatrix@r{
			1\ar@(dl,ul)\ar@/^0.63pc/[r]& 2\ar@/^0.63pc/[d]\\
			4\ar@/^0.63pc/[u]&  3\ar@/^0.63pc/[l].}
$$
\\	
The regular representation of $\mathrm{gr}(A)$ is as follows.
		
$$\begin{picture}(30.00,60.00)
	\unitlength=1.00mm \special{em:linewidth 0.4pt}
		\linethickness{0.4pt}
		\put(-12,20){$1$} \put(-8,15){$2$} \put(-8,10){$3$}\put(-8,5){$4$}\put(-16,13){$1$}\put(-16,6){$1$}
		\put(-8,0){$1$}
		\put(-4,12){$\oplus$}
		\put(0,20){$2$}
		\put(0,15){$3$} \put(0,10){$4$} \put(0,5){$1$}
		\put(0,0){$2$}
		\put(4,12){$\oplus$}
		\put(8,20){$3$} \put(8,15){$4$}  \put(8,10){$1$} \put(8,5){$2$} \put(8,0){$3$}
		\put(12,12){$\oplus$}
		\put(16,20){$4$} \put(16,15){$1$}  \put(16,10){$2$} \put(16,5){$3$} \put(16,0){$4$}
\end{picture}$$
\\
Note that $\mathrm{gr}(A)=\overline{\mathrm{gr}(A)}=C$. We have that $G'$ is the following Brauer graph
$$
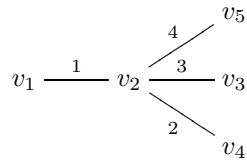
\xymatrix@R-1em{
	& & v_5	&\\
	v_1	\ar@{-}[r]^{1} & v_2\ar@{-}[ur]^{4}\ar@{-}[dr]_{2}\ar@{-}[r]^{3}&v_3 \\
	& & v_4}	
$$
where $m(v_1)=3$, $m(v_2)=2$ and $m(v_3)=m(v_4)=m(v_5)=1$.	
The regular representation of the corresponding Brauer graph algebra $B$ is as follows.
		
		$$\begin{picture}(30.00,110.00)
		\unitlength=1.00mm \special{em:linewidth 0.4pt}
		\linethickness{0.4pt}
		\put(-12,40){$1$} \put(-8,35){$2$} \put(-8,30){$3$}\put(-8,25){$4$}\put(-16,25){$1$}\put(-16,11){$1$}
		\put(-8,20){$1$}\put(-8,15){$2$} \put(-8,10){$3$}\put(-8,5){$4$}
		\put(-12,0){$1$}
		\put(-4,22){$\oplus$}
		\put(0,40){$2$}
		\put(0,35){$3$} \put(0,30){$4$} \put(0,25){$1$}
		\put(0,20){$2$}  \put(0,15){$3$} \put(0,10){$4$} \put(0,5){$1$}\put(0,0){$2$}
		\put(4,22){$\oplus$}
		\put(8,40){$3$} \put(8,35){$4$}  \put(8,30){$1$} \put(8,25){$2$} \put(8,20){$3$} \put(8,15){$4$}  \put(8,10){$1$} \put(8,5){$2$} \put(8,0){$3$}
		\put(12,22){$\oplus$}
		\put(16,40){$4$} \put(16,35){$1$}  \put(16,30){$2$} \put(16,25){$3$} \put(16,20){$4$} \put(16,15){$1$}  \put(16,10){$2$} \put(16,5){$3$} \put(16,0){$4$}
		\end{picture}$$
		
		Note that $\overline{B}=B/\soc(P_1)$ and $C\cong\overline{B}/(\rad^{5}(P_1)\oplus \rad^{5}(P_2)\oplus \rad^{5}(P_3)\oplus \rad^{5}(P_4))$, where $P_i$ is the projective cover of the simple $\overline{B}$-module $S_i$ corresponding to the edge $i$ in $G'$. Since $B$ is not of polynomial growth and $\overline{B}$ is not of polynomial growth, $C$ is not of polynomial growth and therefore $gr(A)$ is not of polynomial growth.					
\end{Ex}

\begin{Lem}
		Let $\overline{\mathrm{gr}(A)}=kQ/I_2$ be defined in $(2.2)$. If $m_0=1$ and $\kappa_1=1$, or  $m_0=2$, $\kappa_0=1$ and $\kappa_1=0$, then $\mathrm{gr}(A)$ is $1$-domestic and therefore the cardinality of $\mathrm{Ba}(\overline{\mathrm{gr}(A)})$ is $1$.
	\end{Lem}
	\begin{proof}
		If $m_0=1$ and $\kappa_1=1$, then there are only two unbalanced edges $v^{(i)}_S\defeqi v^{(i)}_L$ and $v^{(j)}_S\defeqj v^{(j)}_L$ in $G$ such that $j$ is in $G_{i,S}$ and $d_G(v^{(j)}_S,v^{(i)}_S)+1=d_G(v^{(j)}_L, v^{(i)}_S)$.
	Let $i_1<i_2<\cdots<i_t<i_1$ (resp. $ j_1<j_2<\cdots<j_{t_1}<j_1)$ be the cyclic ordering at $v^{(i)}_L$ (resp. $ v^{(j)}_L)$, where $i_1=i$ (resp. $ j_1=j)$ and $t=\mathrm{val}(v^{(i)}_L)$ (resp. $t_1=\mathrm{val}(v^{(j)}_L))$. Note that  $t>2$ and $t_1>2$. Let $E_1=\{i_1,i_2,\ldots,i_t\}$  and $ E_2=\{j_1,j_2,\ldots,j_{t_1}\}$, where the edge $i_k$ is incident to $v^{(i)}_L$ and $v'_k$ for any $2\leq k\leq t$, and the edge $j_k$ is incident to $v^{(j)}_L$ and $w'_k$ for any $2\leq k\leq t_1$. We denote by $E_3$ the set of all unbalanced edges different from $i$ and $j$ in $G$, and by $r_k$ the element in $\mathbb{P}$ corresponding to the unbalanced edge $k$ in $E_3$. There are algebra homomorphisms as follows.
		$$\mathrm{gr}(A)\twoheadrightarrow\overline {\mathrm{gr}(A)}\twoheadrightarrow \overline {\mathrm{gr}(A)}/(\bigoplus\limits_{k\in E_3}
		r_k).$$ 
		
		Since $m_0=1$ and $\kappa_1=1$, we have that  $\mathrm{grd}(u)\geq \mathrm{grd}(v)$ for any edge $u\defeqre v$ different from $i$ and $j$ in $G$ satisfying  $d_G(u,v^{(i)}_S)+1=d_G(v,v^{(i)}_S)$, and the unique walk $[v_{1},a_{1},v_{2},\ldots, v_{k-1}, a_{k-1},v_{k}]$ from $v^{(i)}_S$ to $v^{(j)}_S$ satisfies $\mathrm{grd}(v_1)=\mathrm{grd}(v_2)=\ldots=\mathrm{grd}(v_k)$, where $v_1=v^{(i)}_S$, $v_k=v^{(j)}_S$, and $a_i$ is an edge incident to the vertices $v_i$ and $v_{i+1}$ for each $1\leq i\leq k-1$.
	 Then $G$ satisfies $\star$-condition with respect to any unbalanced edge in $E_3$. Therefore, by Lemma \ref{finite-number}, $\overline {\mathrm{gr}(A)}$ and $\overline {\mathrm{gr}(A)}/(\bigoplus_{k\in E_3}
		r_k)$ have the same representation type. In particular, since $\kappa_0(m_0-1)+\kappa_1=1\neq0$, they are of infinite representation type.
		
 Let $G'=(V(G),E(G),m)$ be a Brauer graph, where $m(v^{(i)}_L)=m(v^{(j)}_L)=2$ and $m(v)=1$ for the other vertices $v$. Let $B$ be the Brauer graph algebra associated with the new Brauer graph $G'$ and $\overline{B}$ the corresponding string algebra defined in $(2.1)$. Note that the quiver of $\overline{B}$ is also $Q$. By the construction of the Brauer graph $G'$ and Theorem \ref{Brauer-graph-domestic}, we have that $B$ and $\overline{B}$ are $1$-domestic. 
 
 Let $E_4=\{i_k|\mathrm{grd}(v^{(i)}_L)\geq \mathrm{grd}(v_k'), 1<\mathrm{val}(v_k'), 2\leq k\leq t\}$ and $E_5=\{j_k|\mathrm{grd}(v^{(j)}_L)\geq \mathrm{grd}(w_k'), 1<\mathrm{val}(w_k'), 2\leq k\leq t_1\}$ (which are also defined in the original Brauer graph $G$), where the edge $i_k$ (resp. $j_k$) is incident to $v^{(i)}_L$ (resp. $v^{(j)}_L$) and $v'_k$ (resp. $w'_k$) in $G$ for any $2\leq k\leq t$ (resp. $2\leq k\leq t_1$). We can get the following algebra isomorphism from their constructions
			$$\overline {\mathrm{gr}(A)}/(\bigoplus\limits_{k\in E_3}
			r_k)\cong \overline{B}/(\bigoplus\limits_{k\in E_1\setminus E_4}\rad^{t+1}(P_k)\oplus\bigoplus\limits_{k\in E_4} \rad^{t}(P_k)\oplus\bigoplus\limits_{k\in E_2\setminus E_5}\rad^{t_1+1}(P_k)\oplus\bigoplus\limits_{k\in E_5} \rad^{t_1}(P_k)),$$
		\\ where $P_{k}$ is the projective cover of the simple $\overline{B}$-module $S_{k}$ corresponding to the edge $k$ in $G'$.
		
Since $\overline{B}$ is $1$-domestic and $\overline {\mathrm{gr}(A)}/(\bigoplus_{k\in E_3}r_k)$ is of infinite representation type, we have that $\overline {\mathrm{gr}(A)}/(\bigoplus_{k\in E_3}r_k$) is $1$-domestic and therefore $\mathrm{gr}(A)$ is $1$-domestic. Hence, the cardinality of $\mathrm{Ba}(\overline{\mathrm{gr}(A)})$ is $1$.	
	
If $m_0=2$, $\kappa_0=1$ and $\kappa_1=0$, then there is only one unbalanced edge $v_S\defeqi v_L$ in $G$ such that $v_0$ is in $G_{i,S}$. Similarly as above, let $i_1<i_2<\cdots<i_t<i_1$ be the cyclic ordering at $v_L$, where $i_1=i$ and $t=\mathrm{grd}(v_L)=\mathrm{val}(v_L)$. Let $E_1=\{i_1,i_2,\ldots,i_t\}$, where the edge $i_k$ is incident to $v_L$ and $v'_k$ for any $2\leq k\leq t$. We denote by $E_2$ the set of all unbalanced edges different from $i$ in $G$, and by $r_k$ the element in $\mathbb{P}$ corresponding to the unbalanced edge $k$ in $E_2$.

Let $G'=(V(G),E(G),m)$ be a Brauer graph, where $m(v_L)=m(v_0)=2$ and $m(v)=1$ for the other vertices $v$. Let $B$ be the Brauer graph algebra associated with the new Brauer graph $G'$ and $\overline{B}$ the corresponding string algebra defined in $(2.1)$. Note that $B$ and $\overline{B}$ are $1$-domestic. 

Let $E_3=\{i_k|\mathrm{grd}(v_L)\geq \mathrm{grd}(v_k'), 1<\mathrm{val}(v_k'), 2\leq k\leq t\}$ (which is also defined in the original Brauer graph $G$), where the edge $i_k$ is incident to $v_L$ and $v'_k$ in $G$ for any $2\leq k\leq t$. We can get the following algebra isomorphism  

$$\overline {\mathrm{gr}(A)}/(\bigoplus\limits_{k\in E_2}
r_k)\cong \overline{B}/(\bigoplus\limits_{k\in E_1\setminus E_3}\rad^{t+1}(P_k)\oplus\bigoplus\limits_{k\in E_3} \rad^{t}(P_k)),$$
\\ where $P_{k}$ is the projective cover of the simple $\overline{B}$-module $S_{k}$ corresponding to the edge $k$ in $G'$. We also have that $\mathrm{gr}(A)$ is $1$-domestic and the cardinality of $\mathrm{Ba}(\overline{\mathrm{gr}(A)})$ is $1$.
\end{proof}

We give an example to illustrate the above result.
\begin{Ex}
Let $G$ be the following Brauer tree with $m_0=2$.		
	$$\xymatrix@R-1em{
		& & \cdot\ar@{-}[r]^{4}	&\cdot\\
		v_0	\ar@{-}[r]^{1} & \cdot\ar@{-}[ur]^{3}\ar@{-}[dr]_{2}& \\
		& & \cdot}	
	$$	
Note that $\kappa_0=1$ and $\kappa_1=0$.
	
Let $A=kQ/I$ be the Brauer tree algebra associated with $G$ and $\mathrm{gr}(A)$ the associated graded algebra of $A$. The quiver $Q$ is as follows.
	
	$$
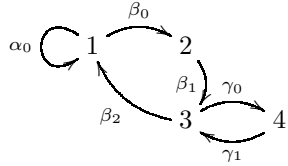
\xymatrix@R-0.8em@r{
		1\ar@/^0.63pc/[r]^{\beta_0}\ar@(ul,dl)_{\alpha_0}& 2\ar@/^0.63pc/[d]_{\beta_1}& \\
		&3\ar@/^0.63pc/[ul]^{\beta_2}\ar@/^0.63pc/[r]^{\gamma_0}&4\ar@/^0.63pc/[l]^{\gamma_1}}
	$$	
The regular representation of $\mathrm{gr}(A)$ is as follows.
	$$\begin{picture}(30.00,50.00)
		\unitlength=1.00mm \special{em:linewidth 0.4pt}
		\linethickness{0.4pt}
		\put(-12,15){$1$} \put(-8,10){$2$} \put(-8,5){$3$}\put(-8,0){$1$}\put(-16,8){$1$}
		\put(-4,9){$\oplus$}
		\put(0,15){$2$}
		\put(0,10){$3$} \put(0,5){$1$} \put(0,0){$2$}
		\put(4,9){$\oplus$}
		\put(12,15){$3$} \put(8,10){$1$}  \put(8,5){$2$} \put(8,0){$3$} \put(16,10){$4$}
		\put(20,9){$\oplus$}
		\put(24,15){$4$} \put(24,10){$3$}   \put(24,5){$4$}
	\end{picture}$$
	\\
Note that $\mathrm{gr}(A)=\overline{\mathrm{gr}(A)}$. Moreover, $b:=\alpha^{-1}_0\beta_2\beta_1\beta_0$ is the unique band in $\overline{\mathrm{gr}(A)}$ (after rotation or taking inverse). We have that $G'$ is the following Brauer graph
	
			$$
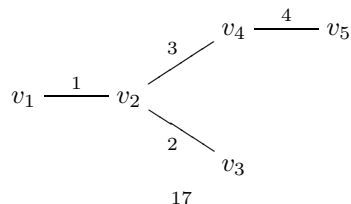
\xymatrix@R-1em{
		& & v_4\ar@{-}[r]^{4}	&v_5\\
		v_1	\ar@{-}[r]^{1} & v_2\ar@{-}[ur]^{3}\ar@{-}[dr]_{2}& \\
		& & v_3}	
	$$	
	where $m(v_1)=2$, $m(v_2)=2$ and $m(v_3)=m(v_4)=m(v_5)=1$.	
The regular representation of the corresponding Brauer graph algebra $B$ is as follows.
$$\begin{picture}(10.00,90.00)
\unitlength=1.00mm \special{em:linewidth 0.4pt}
\linethickness{0.4pt}
\put(-12,30){$1$} \put(-8,25){$2$} \put(-8,20){$3$}\put(-8,15){$1$}\put(-16,25){$1$}
\put(-8,10){$2$}
\put(-8,5){$3$} 
\put(-12,0){$1$}
\put(-4,15){$\oplus$}
\put(0,30){$2$}
\put(0,25){$3$} \put(0,20){$1$} \put(0,15){$2$}
\put(0,10){$3$}  \put(0,5){$1$} \put(0,0){$2$} 
\put(4,15){$\oplus$}
\put(12,30){$3$} \put(8,25){$1$}  \put(8,20){$2$} \put(8,15){$3$} \put(8,10){$1$} \put(8,5){$2$}  \put(12,0){$3$} \put(16,20){$4$} 
\put(20,15){$\oplus$}
\put(24,30){$4$} \put(24,15){$3$}  \put(24,0){$4$} 
\end{picture}$$

Note that $\overline{B}=B/(\soc(P_1)\oplus \soc(P_3))$ and $\overline{\mathrm{gr}(A)}/r_3\cong\overline{B}/(\rad^{4}(P_1)\oplus \rad^{4}(P_2)\oplus \rad^{3}(P_3)),$ where $P_i$ is the projective cover of the simple $\overline{B}$-module $S_i$ corresponding to the edge $i$ in $G'$. Since $B$ is $1$-domestic, $\overline{B}$ is $1$-domestic, and $\overline{\mathrm{gr}(A)}/r_3$ is of infinite representation type, we have that   $\overline{\mathrm{gr}(A)}/r_3$ is $1$-domestic and therefore $\mathrm{gr}(A)$ is $1$-domestic.					
\end{Ex}	

\begin{Lem}\label{Brauer-tree-nondomestic}
Let $\overline{\mathrm{gr}(A)}=kQ/I_2$ defined in $(2.2)$. If $m_0=2$, $\kappa_0\geq2$ or $m_0=2$, $\kappa_1\neq0$, then the cardinality of $\mathrm{Ba}(\overline{\mathrm{gr}(A)})$ is infinite and $\mathrm{gr}(A)$ is not of polynomial growth.
\end{Lem}
\begin{proof}
We have two cases to consider.
		
{\it Case 1.} If $m_0=2$ and $\kappa_0\geq2$, then there are two unbalanced edges $v^{(i)}_{S}\defeqi v^{(i)}_L$ and  $v^{(j)}_S\defeqj v^{(j)}_L$ in $G$ such that the exceptional vertex $v_0$ is a vertex in $G_{i,S}$ and it is also a vertex in $G_{j,S}$. There is a walk $[v_{1},a_{1},v_{2},a_{2},v_{3},\ldots, v_{k-1}, a_{k-1},v_{k}]$ (resp. $[v'_{1},a'_{1},v'_{2},a'_{2},v'_{3},\ldots, v'_{k'-1}, a'_{k'-1},v'_{k'}]$) from $v_0$ to $v^{(i)}_L$ (resp. $v^{(j)}_L$), where  $v_{1}=v_0$, $v_{k}=v^{(i)}_L$, $a_{k-1}=i$ (resp. $v'_1=v_0$, $v'_{k'}=v^{(j)}_L$, $a'_{k'-1}=j$) and $a_{l}$ (resp. $a'_{l}$) is an edge  incident to the vertices $v_{l}$ (resp. $v'_{l}$) and $v_{l+1}$ (resp. $v'_{l+1}$) for each $1\leq l\leq k-1$ (resp. $1\leq l\leq k'-1$).
		
$(a)$ If $a_1=a'_1$, then $Q$ contains the following subquiver
		
$$
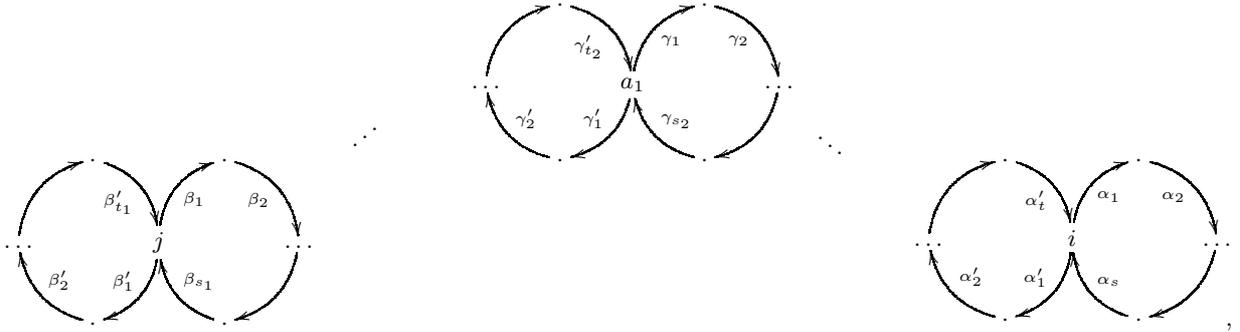
\begin{small}
\xymatrix@C-0.9em@R-0.5em{
				&\cdot \ar@/^0.63pc/[rd]_{\beta_{t_{1}}'} &                                  &  \cdot\ar@/^0.635pc/[rd]_{\beta_{2}}    &  \\
				\ldots\ar@/^0.635pc/[ru]    &       &   j\ar@/^0.63pc/[ru]_{\beta_{1}}\ar@/^0.63pc/[ld]_{\beta_{1}'} &    &\ldots\ar@/^0.63pc/[ld]\\
				&           \cdot\ar@/^0.63pc/[lu]_{\beta_{2}'}     &        & \cdot\ar@/^0.63pc/[lu]_{\beta_{s_{1}}} &
}
\xymatrix@dl@C-2em@R-1em{& & \\
				& \udots &}	
\xymatrix@l@C-0.9em@R-0.5em{
				&\cdot \ar@/^0.63pc/[rd]_{\gamma_{s_2}} &                                  &  \cdot\ar@/^0.635pc/[rd]_{\gamma_{2}'}    &  \\
				\ldots\ar@/^0.635pc/[ru]    &       &   a_1\ar@/^0.63pc/[ru]_{\gamma_{1}'}\ar@/^0.63pc/[ld]_{\gamma_{1}} &    &\ldots\ar@/^0.63pc/[ld]\\
				&           \cdot\ar@/^0.63pc/[lu]_{\gamma_{2}}     &        & \cdot\ar@/^0.63pc/[lu]_{\gamma_{t_2}'} &\quad
}
\xymatrix@dl@C-3em@R-2em{& & \\
		& & \ddots& \\
}
\xymatrix@C-0.8em@R-0.5em{
				&\cdot \ar@/^0.63pc/[rd]_{\alpha'_{t}} &                                  &  \cdot\ar@/^0.635pc/[rd]_{\alpha_{2}}   &  \\
				\ldots\ar@/^0.635pc/[ru]    &       &   i\ar@/^0.63pc/[ru]_{\alpha_{1}}\ar@/^0.63pc/[ld]_{\alpha'_{1}} &    &\ldots\ar@/^0.63pc/[ld]\\
				&           \cdot\ar@/^0.63pc/[lu]_{\alpha'_{2}}     &        & \cdot\ar@/^0.63pc/[lu]_{\alpha_{s}} &\quad,
}
\end{small}$$
where $s=\mathrm{val}(v^{(i)}_S)$, $t=\mathrm{val}(v^{(i)}_L)$, $s_1=\mathrm{val}(v^{(j)}_S)$, $t_1=\mathrm{val}(v^{(j)}_L)$, $s_2=\mathrm{val}(v_2)$, $t_2=\mathrm{val}(v_0)$,   $\alpha_{t}'\ldots\alpha_{1}'$, $\beta_{t_{1}}'\ldots\beta_{1}'$ and $\gamma_{t_2}'\ldots\gamma_{1}'$ are not in $I_2$.
		
There is a simple string $c_{k_1}\ldots c_2 c_1$ (resp. $d_{k'_1}\ldots d_2 d_1$) satisfying $c_1=\gamma^{-1}_{s_2}$ and $t(c_{k_1})=i$ (resp. $d_1=\gamma^{-1}_{s_2}$ and $t(d_{k'_1})=j$).
		
\begin{enumerate}[(1)]
\item If $c_{k_1}$ is an inverse arrow (in other words, $c_{k_1}=\beta_{1}^{-1}$), then $\beta_{t_{1}}'\ldots\beta_{1}'c_{k_1}\ldots c_2 c_{1}\gamma_{t_2}'\cdots\gamma_{1}'$ is also a string. There exists a simple string $c_{k_2}'\ldots c'_2 c_{1}'$ satisfying $c_{1}'=\beta_{s_{1}}^{-1}$ and $t(c'_{k_2})=a_1$. Then	
$$b_1:=c_{k_2}'\ldots  c'_2 c_{1}'\beta_{t_{1}}'\ldots\beta_{1}'c_{k_1}\ldots c_2 c_{1}\gamma_{t_2}'\cdots\gamma_{1}'$$
is a band with source $a_1$.			
\item If $c_{k_1}$ is an arrow (in other words, $c_{k_1}=\beta_{s_{1}}$), then ${(\beta_{1}')}^{-1}$ $\ldots{(\beta_{t_{1}}')}^{-1}c_{k_1}\ldots c_2 c_1\gamma_{t_2}'\ldots\gamma_{1}'$ is also a string. In this situation we can similarly get a band $b_1$ as in $(1)$.
\end{enumerate}
Similarly, we have two cases for $d_{k'_1}$. Then  $b_2:=d_{k'_2}'\ldots d'_2 d_{1}'\alpha'_{t}\ldots\alpha'_2\alpha'_1 d_{k'_1}\ldots d_2 d_{1}\gamma_{t_2}'\cdots\gamma_{1}'$ (or $b_2:=d_{k'_2}'\ldots d'_2 d_{1}'\\(\alpha'_{1})^{-1}\ldots(\alpha'_t)^{-1} d_{k'_1}\ldots d_2 d_{1}\gamma_{t_2}'\cdots\gamma_{1}'$) is a band with source $a_1$.
		
$(b)$ If $a_1\neq a'_1$, then $d_G(v^{(j)}_S,v^{(i)}_S)+1=d_G(v^{(j)}_L,v^{(i)}_S)$  and $j$ is in $G_{i,S}$. Therefore $(i,j)$ is an unbalanced edge pair. We have that $Q$ contains the following subquivers
		
$$
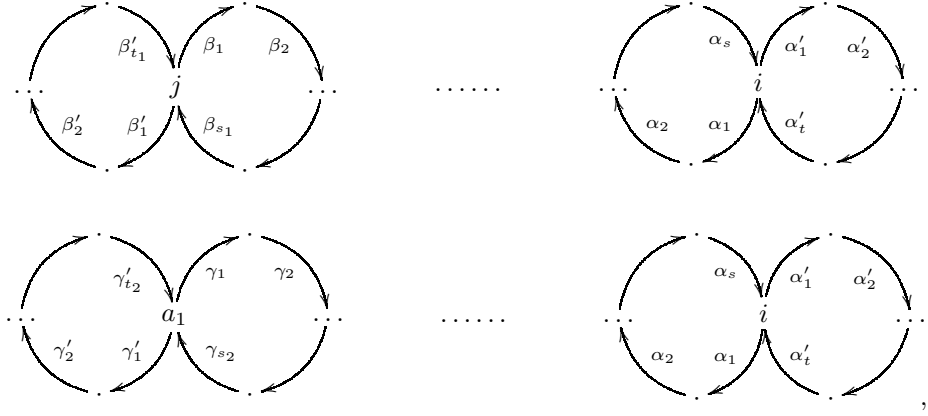
\xymatrix@C-0.8em@R-0.5em@r{
			&\cdot \ar@/^0.63pc/[rd]_{\beta_{t_{1}}'} &                                  &  \cdot\ar@/^0.635pc/[rd]_{\beta_{2}}    &  \\
			\ldots\ar@/^0.635pc/[ru]    &       &   j\ar@/^0.63pc/[ru]_{\beta_{1}}\ar@/^0.63pc/[ld]_{\beta_{1}'} &    &\ldots\ar@/^0.63pc/[ld]\\
			&           \cdot\ar@/^0.63pc/[lu]_{\beta_{2}'}     &        & \cdot\ar@/^0.63pc/[lu]_{\beta_{s_{1}}} &\
}
\xymatrix{& & \\
			& \ldots\ldots &}
\xymatrix@C-0.8em@R-0.5em@r{
			&\cdot \ar@/^0.63pc/[rd]_{\alpha_{s}} &                                  &  \cdot\ar@/^0.635pc/[rd]_{\alpha_{2}'}    &  \\
			\ldots\ar@/^0.635pc/[ru]    &       &   i\ar@/^0.63pc/[ru]_{\alpha_{1}'}\ar@/^0.63pc/[ld]_{\alpha_{1}} &    &\ldots\ar@/^0.63pc/[ld]\\
			&           \cdot\ar@/^0.63pc/[lu]_{\alpha_{2}}     &        & \cdot\ar@/^0.63pc/[lu]_{\alpha_{t}'} &\quad
}$$
$$\xymatrix@C-0.8em@R-0.5em@r{
			&\cdot \ar@/^0.63pc/[rd]_{\gamma_{t_{2}}'} &                                  &  \cdot\ar@/^0.635pc/[rd]_{\gamma_{2}}    &  \\
			\ldots\ar@/^0.635pc/[ru]    &       &   a_1\ar@/^0.63pc/[ru]_{\gamma_{1}}\ar@/^0.63pc/[ld]_{\gamma_{1}'} &    &\ldots\ar@/^0.63pc/[ld]\\
			&           \cdot\ar@/^0.63pc/[lu]_{\gamma_{2}'}     &        & \cdot\ar@/^0.63pc/[lu]_{\gamma_{s_{2}}} &
}
\xymatrix{& & \\
			& \ldots\ldots &}
\xymatrix@C-0.8em@R-0.5em@r{
			&\cdot \ar@/^0.63pc/[rd]_{\alpha_{s}} &                                  &  \cdot\ar@/^0.635pc/[rd]_{\alpha_{2}'}    &  \\
			\ldots\ar@/^0.635pc/[ru]    &       &   i\ar@/^0.63pc/[ru]_{\alpha_{1}'}\ar@/^0.63pc/[ld]_{\alpha_{1}} &    &\ldots\ar@/^0.63pc/[ld]\\
			&           \cdot\ar@/^0.63pc/[lu]_{\alpha_{2}}     &        & \cdot\ar@/^0.63pc/[lu]_{\alpha_{t}'} &\quad ,
}$$
where $s=\mathrm{val}(v^{(i)}_S)$, $t=\mathrm{val}(v^{(i)}_L)$,  $s_1=\mathrm{val}(v^{(j)}_S)$, $t_1=\mathrm{val}(v^{(j)}_L)$, $t_2=\mathrm{val}(v_0)$, $s_2=\mathrm{val}(v_2)$,  $\alpha_{t}'\ldots\alpha_{1}'$, $\beta_{t_{1}}'\ldots\beta_{1}'$ and $\gamma_{t_2}'\ldots\gamma_{1}'$ are not in $I_2$.
		
There is a simple string $c_{k_1}\ldots c_2 c_1$ (resp. $d_{k'_1}\ldots d_2 d_1$) satisfying $c_1=\alpha^{-1}_{s}$ and $t(c_{k_1})=j$ (resp. $d_1=\alpha^{-1}_s$ and $t(d_{k'_1})=a_1$).
		
\begin{enumerate}[(1)]
\item If $c_{k_1}$ is an inverse arrow (in other words, $c_{k_1}=\beta_{1}^{-1}$), then $\beta_{t_{1}}'\ldots\beta_{1}'c_{k_1}\ldots c_2 c_{1}\alpha'_{t}\cdots\alpha'_{1}$ is also a string. There exists a simple string $c_{k_2}'\ldots c'_2 c_{1}'$ satisfying $c_{1}'=\beta_{s_{1}}^{-1}$ and $t(c'_{k_2})=i$. Then	
$$b_1:=c_{k_2}'\ldots  c'_2 c_{1}'\beta_{t_{1}}'\ldots\beta_{1}'c_{k_1}\ldots c_2 c_{1}\alpha_{t}'\cdots\alpha_{1}'$$
is a band with source $i$.
\item If $c_{k_1}$ is an arrow (in other words, $c_{k_1}=\beta_{s_{1}}$), then ${(\beta_{1}')}^{-1}$ $\ldots{(\beta_{t_{1}}')}^{-1}c_{k_1}\ldots c_2 c_1\alpha'_{t}\cdots\alpha'_{1}$ is also a string. In this situation we can similarly get a band $b_1$ as in $(1)$.
\end{enumerate}

Similarly, we have two cases for $d_{k'_1}$. Then  $b_2:=d_{k'_2}'\ldots d'_2 d_{1}'\gamma'_{t_2}\ldots\gamma'_2\gamma'_1 d_{k'_1}\ldots d_2 d_{1}\alpha_{t}'\cdots\alpha_{1}'$ (or $b_2:=d_{k'_2}'\ldots d'_2 d_{1}'\\(\gamma'_{1})^{-1}\ldots(\gamma'_{t_2})^{-1} d_{k'_1}\ldots d_2 d_{1}\alpha_{t}'\cdots\alpha_{1}'$) is a band with source $i$.
		
In either case, we have that two distinct bands $b_1$ and $b_2$ in $\overline{\mathrm{gr}(A)}$ and  $b_1$ and $b_2$ satisfy the condition of Lemma \ref{nondomestic} by construction. Therefore, the cardinality of $\mathrm{Ba}(\overline{\mathrm{gr}(A)})$ is infinite and $\mathrm{gr}(A)$ is not of polynomial growth.
		
{\it Case 2.} If  $m_0=2$ and $\kappa_1\neq0$, by Lemma \ref{e_1}, then there are two unbalanced edges $v^{(i)}_S\defeqi v^{(i)}_L$ and  $v^{(j)}_S\defeqj v^{(j)}_L$ in $G$ such that $v_0$ is in $G_{i,S}$ and $(i,j)$ is an unbalanced edge pair. It is similar to the above case $(b)$. We still get our desired result. 
\end{proof}
We give an example to illustrate the above result.
\begin{Ex}
Let  $G$ be the following Brauer tree with $m_0=2$.	
$$
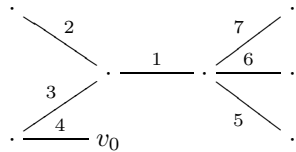
\xymatrix@R-1em{
			\cdot\ar@{-}[dr]^{2}	& & & \cdot& 	\\
			&  \cdot \ar@{-}[r]^{1} &\cdot \ar@{-}[r]^{6}\ar@{-}[ur]^{7}\ar@{-}[dr]_{5}&\cdot & \\
			\cdot\ar@{-}[ur]^{3}\ar@{-}[r]^{4}	&v_0 & &\cdot &  & }
$$		

Let $A=kQ/I$ be the Brauer graph algebra associated with $G$ and $\mathrm{gr}(A)$ the associated graded algebra of $A$. The quiver $Q$ is as follows.	
$$
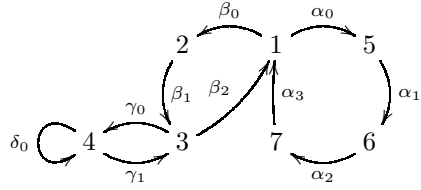
\xymatrix@r{
			& 2\ar@/_0.63pc/[d]^{\beta_1} & 1\ar@/_0.63pc/[l]_{\beta_0} \ar@/^0.63pc/[r]^{\alpha_0} &5\ar@/^0.63pc/[d]^{\alpha_1} &\\
			4\ar@(ul,dl)_{\delta_0}\ar@/_0.63pc/[r]_{\gamma_1}	& 3\ar@/_0.63pc/[l]_{\gamma_0} \ar@/_0.4pc/[ur]^{\beta_2} & 7\ar@/^0.1pc/[u]_{\alpha_3} &
			6\ar@/^0.63pc/[l]^{\alpha_2} &  &}
$$	

We have that  $b_1=\gamma^{-1}_1\beta_1\beta_0\beta_2\gamma^{-1}_0\delta_0$ and $b_2=\gamma^{-1}_1\beta_1\beta_0\alpha^{-1}_0\alpha^{-1}_1\alpha^{-1}_2\alpha^{-1}_3\beta_2\gamma^{-1}_0\delta_0$ are bands in $\overline{\mathrm{gr}(A)}$
\end{Ex}

\begin{Lem}
Let $\overline{\mathrm{gr}(A)}=kQ/I_2$ be defined in $(2.2)$. If\ $\kappa_1\geq2$, then the cardinality of $\mathrm{Ba}(\overline{\mathrm{gr}(A)})$ is infinite and $\mathrm{gr}(A)$ is not of polynomial growth.
\end{Lem}	
\begin{proof}
Since  $\kappa_1\geq2$, by Lemma \ref{e_1-character}, there are three unbalanced edges $v^{(i)}_S\defeqi v^{(i)}_L$, $v^{(j)}_S\defeqj v^{(j)}_L$ and  $v^{(k)}_S\defeqk v^{(k)}_L$ in $G$ such that $(i,j)$ and $(i,k)$ are unbalanced edge pairs. Then $Q$ contains the following subquiver
$$
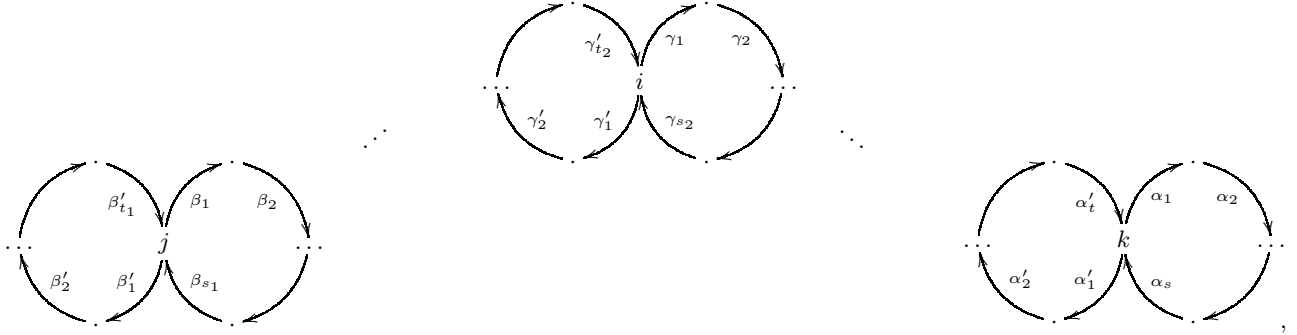
\begin{small}
\xymatrix@C-0.8em@R-0.5em{
				&\cdot \ar@/^0.63pc/[rd]_{\beta_{t_{1}}'} &                                  &  \cdot\ar@/^0.635pc/[rd]_{\beta_{2}}    &  \\
				\ldots\ar@/^0.635pc/[ru]    &       &   j\ar@/^0.63pc/[ru]_{\beta_{1}}\ar@/^0.63pc/[ld]_{\beta_{1}'} &    &\ldots\ar@/^0.63pc/[ld]\\
				&           \cdot\ar@/^0.63pc/[lu]_{\beta_{2}'}     &        & \cdot\ar@/^0.63pc/[lu]_{\beta_{s_{1}}} &
}
\xymatrix@dl@C-2em@R-1em{& & \\
				& \udots &}	
\xymatrix@l@C-0.8em@R-0.5em{
				&\cdot \ar@/^0.63pc/[rd]_{\gamma_{s_2}} &                                  &  \cdot\ar@/^0.635pc/[rd]_{\gamma_{2}'}    &  \\
				\ldots\ar@/^0.635pc/[ru]    &       &   i\ar@/^0.63pc/[ru]_{\gamma_{1}'}\ar@/^0.63pc/[ld]_{\gamma_{1}} &    &\ldots\ar@/^0.63pc/[ld]\\
				&           \cdot\ar@/^0.63pc/[lu]_{\gamma_{2}}     &        & \cdot\ar@/^0.63pc/[lu]_{\gamma_{t_2}'} &\quad
}
\xymatrix@dl@C-2em@R-1em{& & \\
				&\ddots &
}
\xymatrix@C-0.8em@R-0.5em{
				&\cdot \ar@/^0.63pc/[rd]_{\alpha'_{t}} &                                  &  \cdot\ar@/^0.635pc/[rd]_{\alpha_{2}}   &  \\
				\ldots\ar@/^0.635pc/[ru]    &       &   k\ar@/^0.63pc/[ru]_{\alpha_{1}}\ar@/^0.63pc/[ld]_{\alpha'_{1}} &    &\ldots\ar@/^0.63pc/[ld]\\
				&           \cdot\ar@/^0.63pc/[lu]_{\alpha'_{2}}     &        & \cdot\ar@/^0.63pc/[lu]_{\alpha_{s}} &\quad,
}
\end{small}$$
where $s=\mathrm{val}(v^{(k)}_S)$, $t=\mathrm{val}(v^{(k)}_L)$, $s_1=\mathrm{val}(v^{(j)}_S)$, $t_1=\mathrm{val}(v^{(j)}_L)$, $s_2=\mathrm{val}(v^{(i)}_S)$, $t_2=\mathrm{val}(v^{(i)}_L)$, $\alpha_{t}'\ldots\alpha_{1}'$, $\beta_{t_{1}}'\ldots\beta_{1}'$ and $\gamma_{t_2}'\ldots\gamma_{1}'$ are not in $I_2$.
		
There is a simple string $c_{k_1}\ldots c_2 c_1$ (resp. $d_{k'_1}\ldots d_2 d_1$) satisfying $c_1=\gamma^{-1}_{s_2}$ and $t(c_{k_1})=j$ (resp. $d_1=\gamma^{-1}_{s_2}$ and $t(d_{k'_1})=k$).
		
\begin{enumerate}[(1)]
\item If $c_{k_1}$ is an inverse arrow (in other words, $c_{k_1}=\beta_{1}^{-1}$), then $\beta_{t_{1}}'\ldots\beta_{1}'c_{k_1}\ldots c_2 c_{1}\gamma_{t_2}'\cdots\gamma_{1}'$ is also a string. There exists a simple string $c_{k_2}'\ldots c'_2 c_{1}'$ satisfying $c_{1}'=\beta_{s_{1}}^{-1}$ and $t(c'_{k_2})=i$. Then $$b_1:=c_{k_2}'\ldots  c'_2 c_{1}'\beta_{t_{1}}'\ldots\beta_{1}'c_{k_1}\ldots c_2 c_{1}\gamma_{t_2}'\cdots\gamma_{1}'$$ is a band with source $i$.
			
\item If $c_{k_1}$ is an arrow (in other words, $c_{k_1}=\beta_{s_{1}}$), then ${(\beta_{1}')}^{-1}$ $\ldots{(\beta_{t_{1}}')}^{-1}c_{k_1}\ldots c_2 c_1\gamma_{t_2}'\ldots\gamma_{1}'$ is also a string. In this situation we can similarly get a band $b_1$ as in $(1)$.
\end{enumerate}

Similarly, we have two cases for $d_{k'_1}$. Then  $b_2:=d_{k'_2}'\ldots d'_2 d_{1}'\alpha'_{t}\ldots\alpha'_2\alpha'_1 d_{k'_1}\ldots d_2 d_{1}\gamma_{t_2}'\cdots\gamma_{1}'$ (or $b_2:=d_{k'_2}'\ldots d'_2 d_{1}'\\(\alpha'_1)^{-1}\ldots(\alpha'_t)^{-1} d_{k'_1}\ldots d_2 d_{1}\gamma_{t_2}'\cdots\gamma_{1}'$) is a band with source $i$. Moreover, $b_1$ and $b_2$ satisfy the condition of Lemma \ref{nondomestic} by construction. Therefore, the cardinality of $\mathrm{Ba}(\overline{\mathrm{gr}(A)})$ is infinite and $\mathrm{gr}(A)$ is not of polynomial growth.
\end{proof}
	
By the above results, we have the following characterization of domestic representation type of $\mathrm{gr}(A)$.
	
\begin{Prop}\label{Brauer-tree-domestic}
Let $A$ be the Brauer tree algebra associated with a Brauer tree with an exceptional vertex $v_0$ of multiplicity  $m_0$ and $\mathrm{gr}(A)$ the graded algebra associated with the radical filtration of $A$. Then the following are equivalent.
\begin{enumerate}[$(1)$]
    \item	$\mathrm{gr}(A)$ is of polynomial growth.
	\item $\mathrm{gr}(A)$ is domestic.
	\item $\mathrm{gr}(A)$ is $1$-domestic.
	\item $\kappa_0(m_0-1)+\kappa_1=1$.
	\item The cardinality of $\mathrm{Ba}(\overline{\mathrm{gr}(A)})$ is finite, where  $\overline{\mathrm{gr}(A)}$ is defined in $(2.2)$.
\end{enumerate}
\end{Prop}

\section{The case that $G$ is a tree with $m(v)=2$ for exactly two vertices and $m(v)=1$ for other vertices}

In this section, we describe when $\mathrm{gr}(A)=kQ/I'$ is domestic under the assumption that $G=(V(G), E(G), m)$ is a tree with $m(v)=2$ for exactly two vertices $v=w_0,w_1\in V(G)$ and $m(v)=1$ for all $v\in V(G)$, $v\neq w_0,w_1$, where $V(G)$ is the vertex set, $E(G)$ is the edge set and $m$ is the multiplicity function of $G$.
	
\begin{Prop}\label{tree-not-polynomial}
Let $\overline{\mathrm{gr}(A)}=kQ/I_2$ be defined in $(2.2)$.
If there is an unbalanced edge $v_S\defeqi v_L$ in $G$ such that $w_0$ and $w_1$ are in $G_{i,S}$. Then $\overline{\mathrm{gr}(A)}$ and $\mathrm{gr}(A)$ are not of polynomial growth.
\end{Prop}
	
\begin{proof}
Using an approach similar to the proof of Proposition \ref{Brauer-tree-not-polynomial}, we have that $\overline{\mathrm{gr}(A)}$ and $\mathrm{gr}(A)$ are not of polynomial growth.
\end{proof}
	
\begin{Lem}\label{tree-nondomestic}
Let $\overline{\mathrm{gr}(A)}=kQ/I_2$ be defined in $(2.2)$. If there is an unbalanced edge $v_S\defeqi v_L$ in $G$ such that $w_0$ and $w_1$ are in different connected branch of  $G\setminus i$. Then the cardinality of $\mathrm{Ba}(\overline{\mathrm{gr}(A)})$ is infinite and $\mathrm{gr}(A)$ is not of polynomial growth.
\end{Lem}
	
\begin{proof}
Without loss of generality, we assume that $w_0$ is in $G_{i,S}$ and $w_1$ is in $G_{i,L}$. We consider the walk $[v_{1},a_{1},v_{2},\ldots, v_{k-1}, a_{k-1},v_{k}]$ from $w_0$ to $w_1$, where  $v_{1}=w_0$, $v_{k}=w_1$. There is an edge $v_j\defeqtj v_{j+1}$ in the walk such that $a_j=i$, $v_j=v_S$ and $v_{j+1}=v_L$. Then $Q$ contains the following subquiver
		
$$
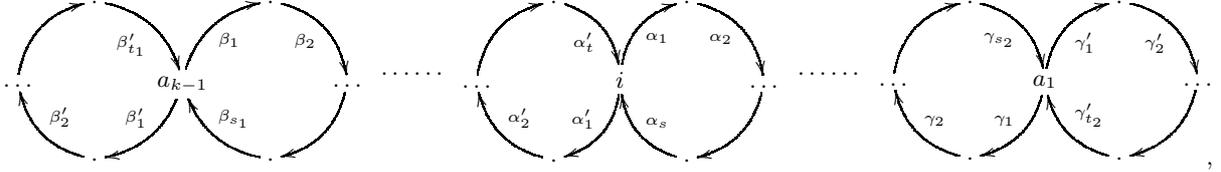
\begin{small}		
\xymatrix@r@C-0.8em@R-0.5em{
				&\cdot \ar@/^0.63pc/[rd]_{\beta_{t_{1}}'} &                                  &  \cdot\ar@/^0.635pc/[rd]_{\beta_{2}}    &  \\
				\ldots\ar@/^0.635pc/[ru]    &       &   a_{k-1}\ar@/^0.63pc/[ru]_{\beta_{1}}\ar@/^0.63pc/[ld]_{\beta_{1}'} &    &\ldots\ar@/^0.63pc/[ld]\\
				&           \cdot\ar@/^0.63pc/[lu]_{\beta_{2}'}     &        & \cdot\ar@/^0.63pc/[lu]_{\beta_{s_{1}}} &
}
\xymatrix@R-0.5em@C-3em{& & \\
				& \ldots\ldots &}
\xymatrix@r@C-0.8em@R-0.5em{
				&\cdot \ar@/^0.63pc/[rd]_{\alpha'_{t}} &                                  &  \cdot\ar@/^0.635pc/[rd]_{\alpha_{2}}   &  \\
				\ldots\ar@/^0.635pc/[ru]    &       &   i\ar@/^0.63pc/[ru]_{\alpha_{1}}\ar@/^0.63pc/[ld]_{\alpha'_{1}} &    &\ldots\ar@/^0.63pc/[ld]\\
				&           \cdot\ar@/^0.63pc/[lu]_{\alpha'_{2}}     &        & \cdot\ar@/^0.63pc/[lu]_{\alpha_{s}} &\quad
}
\xymatrix@R-0.5em@C-3em{& & \\
				& \ldots\ldots &}
\xymatrix@r@C-0.8em@R-0.5em{
				&\cdot \ar@/^0.63pc/[rd]_{\gamma_{s_2}} &                                  &  \cdot\ar@/^0.635pc/[rd]_{\gamma_{2}'}    &  \\
				\ldots\ar@/^0.635pc/[ru]    &       &   a_1\ar@/^0.63pc/[ru]_{\gamma_{1}'}\ar@/^0.63pc/[ld]_{\gamma_{1}} &    &\ldots\ar@/^0.63pc/[ld]\\
				&           \cdot\ar@/^0.63pc/[lu]_{\gamma_{2}}     &        & \cdot\ar@/^0.63pc/[lu]_{\gamma_{t_2}'} &\quad ,
}
\end{small}$$
where $s=\mathrm{val}(v_S)$, $t=\mathrm{val}(v_L)$, $s_1=\mathrm{val}(v_{k-1})$, $t_1=\mathrm{val}(w_1)$, $s_2=\mathrm{val}(v_2)$, $t_2=\mathrm{val}(w_0)$,  $\alpha_{t}'\ldots\alpha_{1}'$, $\beta_{t_{1}}'\ldots\beta_{1}'$ and $\gamma_{t_2}'\ldots\gamma_{1}'$ are not in $I_2$.
		
There is a simple string $c_{k_1}\ldots c_2 c_1$ (resp. $d_{k'_1}\ldots d_2 d_1$) satisfying $c_1=\gamma^{-1}_{s_2}$ and $t(c_{k_1})=a_{k-1}$ (resp. $d_1=\gamma^{-1}_{s_2}$ and $t(d_{k'_1})=i$).
		
\begin{enumerate}[$(1)$]
\item If $c_{k_1}$ is an inverse arrow (in other words, $c_{k_1}=\beta_{1}^{-1}$), then $\beta_{t_{1}}'\ldots\beta_{1}'c_{k_1}\ldots c_2 c_{1}\gamma_{t_2}'\cdots\gamma_{1}'$ is also a string. There exists a simple string $c_{k_2}'\ldots c'_2 c_{1}'$ satisfying $c_{1}'=\beta_{s_{1}}^{-1}$ and $t(c'_{k_2})=a_1$. Then	
$$b_1:=c_{k_2}'\ldots  c'_2 c_{1}'\beta_{t_{1}}'\ldots\beta_{1}'c_{k_1}\ldots c_2 c_{1}\gamma_{t_2}'\cdots\gamma_{1}'$$ is a band with source $a_1$.
\item If $c_{k_1}$ is an arrow (in other words, $c_{k_1}=\beta_{s_{1}}$), then ${(\beta_{1}')}^{-1}$ $\ldots{(\beta_{t_{1}}')}^{-1}c_{k_1}\ldots c_2 c_1\gamma_{t_2}'\ldots\gamma_{1}'$ is also a string. In this situation we can similarly get a band $b_1$ as in $(1)$.
\end{enumerate}

Similarly, we have two cases for $d_{k'_1}$. Then  $b_2:=d_{k'_2}'\ldots d'_2 d_{1}' \alpha'_{t}\ldots\alpha'_2\alpha'_1d_{k'_1}\ldots d_2 d_{1}\gamma_{t_2}'\cdots\gamma_{1}'$ (or $b_2:=d_{k'_2}'\ldots d'_2 d_{1}'\\ (\alpha'_1)^{-1}\ldots(\alpha'_t)^{-1} d_{k'_1}\ldots d_2 d_{1}\gamma_{t_2}'\cdots\gamma_{1}'$) is a band with source $a_1$. Moreover, $b_1$ and $b_2$ satisfy the condition of Lemma \ref{nondomestic} by construction. Therefore, the cardinality of $\mathrm{Ba}(\overline{\mathrm{gr}(A)})$ is infinite and $\mathrm{gr}(A)$ is not of polynomial growth. 
\end{proof}
	
By Proposition \ref{1-domestic-tree}, Proposition \ref{tree-not-polynomial} and Lemma \ref{tree-nondomestic}, we have the following characterization of domestic representation type of $\mathrm{gr}(A)$.
	
\begin{Prop}\label{tree-domestic}
Let $A$ be the Brauer graph algebra associated with a Brauer graph $G$ which is a tree with $m(v)=2$ for exactly two vertices $v=w_0,w_1$ and $m(v)=1$ for all $v\neq w_0,w_1$, and $\mathrm{gr}(A)$ the graded algebra associated with the radical filtration of $A$. Then the following are equivalent.
	\begin{enumerate}[$(1)$]
	    \item $\mathrm{gr}(A)$ is of polynomial growth.
		\item $\mathrm{gr}(A)$ is domestic.
		\item $\mathrm{gr}(A)$ is $1$-domestic.
		\item There is no unbalanced edge in $G$ or $w_0$ and $w_1$ are in $G_{i,L}$ for any unbalanced edge $v_S\defeqi v_L$ in $G$ (In other words, $G$ satisfies $\star$-condition with respect to any unbalanced edge in $G$).
		\item The cardinality of $\mathrm{Ba}(\overline{\mathrm{gr}(A)})$ is finite, where  $\overline{\mathrm{gr}(A)}$ is defined in $(2.2)$.
\end{enumerate}
\end{Prop}

We can describe when $\mathrm{gr}(A)$ is domestic from the graded degrees of vertices in $G$ point of view in the following
	
\begin{Prop}\label{tree-another}
Let $A$ be the Brauer graph algebra associated with a Brauer graph $G$ which is a tree with $m(v)=2$ for exactly two vertices $v=w_0,w_1$ and $m(v)=1$ for all $v\neq w_0,w_1$, and $\mathrm{gr}(A)$ the associated graded algebra of $A$. Then $\mathrm{gr}(A)$ is domestic if and only if it satisfies the following conditions.
\begin{enumerate}[$(1)$]
	\item  $\mathrm{grd}(w_0)=\mathrm{grd}(w_1)$.
	\item Any walk from $w_0$ (or from $w_1$) is degree decreasing. 	
\end{enumerate}
\end{Prop}
	
\begin{proof}
``$\Longrightarrow$"
Suppose on the contrary that $\mathrm{grd}(w_0)\neq \mathrm{grd}(w_1)$. Consider the walk $[v_{1},a_{1},v_{2},\ldots,v_{k-1},a_{k-1},v_{k}]$ from $w_0$ to $w_1$, where $v_1=w_0$, $v_k=w_1$, we have that there is an  unbalanced edge $v_{i}\defeqt v_{i+1}$ with $\mathrm{grd}(v_i)\neq \mathrm{grd}(v_{i+1})$ for some $1\leq i\leq k-1$ in the walk. Without loss of generality, we assume that $\mathrm{grd}(v_i)<\mathrm{grd}(v_{i+1})$. Then $w_0$ is in $G_{a_i,S}$ and $w_1$ is in $G_{a_i,L}$, by Proposition \ref{tree-domestic}, $\mathrm{gr}(A)$ is nondomestic which is a contradiction. Therefore $\mathrm{grd}(w_0)=\mathrm{grd}(w_1)$ and the condition $(1)$ holds.

In order to verify the condition (2). Suppose on the contrary that there is a vertex $w'$ in $G$ such that the walk $[v_{1},a_{1},v_{2},a_{2},v_{3},\ldots,v_{k-1},a_{k-1},v_{k}]$ from $w_0$ to $w'$  is not degree decreasing, where $v_1=w_0$ and $v_k=w'$. In other words, there exists an unbalanced edge $v_{i}\defeqt v_{i+1}$ with $\mathrm{grd}(v_{i})<\mathrm{grd}(v_{i+1})$ for some $1\leq i\leq k-1$ in the walk. We have that $w_0$ is in $G_{a_i,S}$. Moreover, since $\mathrm{gr}(A)$ is domestic, by Proposition \ref{tree-domestic}, $w_0$ is in $G_{a_i,L}$. A contradiction. 
		
``$\Longleftarrow$"  We suppose on the contrary that $\mathrm{gr}(A)$ is nondomestic. By Proposition \ref{isomorphism} and Proposition \ref{tree-domestic}, we have that there is some unbalanced edge $v_S\defeqi v_L$ such that $w_0$, $w_1$ are in $G_{i,S}$ or $w_0$, $w_1$ are in different connected branch of $G\setminus i$.
		
{\it Case 1.} If $w_0$, $w_1$ are in $G_{i,S}$. Consider the walk $[v_{1},a_{1},v_{2},\ldots,v_{k-1},a_{k-1},v_{k}]$ from $w_0$ to $v_L$, where $v_1=w_0$ and $v_{k}=v_L$. Then $i=a_{k-1}$. Since the above walk is degree decreasing, we have $\mathrm{grd}(v_S)\geq \mathrm{grd}(v_L)$, which is clearly a contradiction.	
	
{\it Case 2.} If $w_0$, $w_1$ are in different connected branch of $G\setminus i$. Consider the walk $[v_{1},a_{1},v_{2},\ldots,v_{k-1},a_{k-1},v_{k}]$ from $w_0$ to $w_1$, where $v_1=w_0$ and $v_k=w_1$. Then $i$ is an edge in the walk. Since the above walk is degree decreasing, we have  $\mathrm{grd}(w_0)\neq \mathrm{grd}(w_1)$. It contradicts the condition $(1)$.
\end{proof}

\section{The case that $G$ is a graph with a unique cycle and $m\equiv 1$}

In this section, we describe when $\mathrm{gr}(A)=kQ/I'$ is domestic in the case that $G$ is a graph with a unique cycle and $m(v)=1$ for any vertex $v$ in $G$.
	
\begin{Prop}\label{cycle-not-polynomial}
Let $\overline{\mathrm{gr}(A)}=kQ/I_2$ be defined in $(2.2)$.
If there is an unbalanced edge $v_S\defeqi v_L$ which is not an edge in the unique cycle such that the unique cycle is in $G_{i,S}$. Then $\overline{\mathrm{gr}(A)}$  and $\mathrm{gr}(A)$ are not of polynomial growth.
\end{Prop}
	
\begin{proof}
Using an approach similar to the proof of Proposition \ref{Brauer-tree-not-polynomial}, we have that $\overline{\mathrm{gr}(A)}$ and $\mathrm{gr}(A)$ are not of polynomial growth.  		
\end{proof}
	
\begin{Lem}\label{greater-than-1}
If some edges in the unique cycle are unbalanced edges, then the number of unbalanced edges in the unique cycle is greater than $1$. Precisely, if there is an unbalanced edge $v^{(i)}_S\defeqi v^{(i)}_L$ in the unique cycle, then there is another unbalanced edge $v^{(j)}_S\defeqj v^{(j)}_L$ with $v^{(i)}_S\neq v^{(j)}_L$ in the unique cycle such that there is a walk $[v_{1},a_{1},v_{2},\ldots, v_{k-1}, a_{k-1},v_{k}]$ from $v^{(i)}_S$ to $v^{(j)}_L$ satisfying $i\neq a_1$ and $a_{k-1}=j$, where $v_{1}=v^{(i)}_S$, $v_{k}=v^{(j)}_L$ and $a_{l}$ is an edge in the unique cycle incident to the vertices $v_{l}$ and $v_{l+1}$ for each $1\leq l\leq k-1$.
\end{Lem}
\begin{proof}
Note that the unique cycle is connected, if there exists an unbalanced edge in the unique cycle, then the number of  unbalanced edges in the unique cycle is greater than $1$.

For any unbalanced edge $v^{(i)}_S\defeqi v^{(i)}_L$ in the unique cycle, there is also an unbalanced edge  $v^{(j')}_S\defeqjk v^{(j')}_L$ different from $i$ in the unique cycle. If $v^{(i)}_S= v^{(j')}_L$, by the connectivity of cycle, then there is an unbalanced edge $v^{(j)}_S\defeqj v^{(j)}_L$ different from $i$ and $j'$ in the unique cycle. Therefore,  $v^{(i)}_S\neq v^{(j)}_L$.

The above shows that for any unbalanced edge $v^{(i)}_S\defeqi v^{(i)}_L$ in the unique cycle, we have another unbalanced edge $v^{(j)}_S\defeqj v^{(j)}_L$ satisfying $v^{(i)}_S\neq v^{(j)}_L$ in the unique cycle.  There is a walk $[v_{1},a_{1},v_{2},\ldots, v_{k-1}, a_{k-1},v_{k}]$ from $v^{(i)}_S$ to $v^{(j)}_L$ satisfying $i\neq a_1$, where $v_{1}=v^{(i)}_S$, $v_{k}=v^{(j)}_L$ and $a_{l}$ is an edge in the unique cycle incident to the vertices $v_{l}$ and $v_{l+1}$ for each $1\leq l\leq k-1$. Moreover, there is also a walk $[v'_{1},a'_{1},v'_{2},\ldots, v'_{k'-1}, a'_{k'-1},v'_{k'}]$ from $v^{(i)}_S$ to $v^{(j)}_L$ different from the above walk, where  $v'_{1}=v^{(i)}_S$, $v'_{k'}=v^{(j)}_L$, $a'_1=i$ and $a'_{l}$ is an edge in the unique cycle incident to the vertices $v'_{l}$ and $v'_{l+1}$ for each $1\leq l\leq k'-1$. We have the following two cases for $a_{k-1}$.

\begin{enumerate}[$(a)$]
	\item If $a_{k-1}=j$, then $j$ and the walk $[v_{1},a_{1},v_{2},\ldots, v_{k-1}, a_{k-1},v_{k}]$ give our desired result.
    \item If $a_{k-1}\neq j$, then $a'_{k'-1}=j$ and there are two cases to be considered.
\begin{enumerate}[$(1)$]
\item If $\mathrm{grd}(v^{(j)}_S)<\mathrm{grd}(v^{(i)}_L)$, by the connectivity of cycle, then there is an unbalanced edge $a'_t$ in the walk $[a'_{1}, a'_2, \ldots,  a'_{k'-1}]$ satisfying  $\mathrm{grd}(v'_{t})>\mathrm{grd}(v'_{t+1})$ and $v'_{t}\neq v^{(i)}_S$. The unbalanced edge $a'_t$ and the walk $[v_{1},a_{1},v_{2},\ldots, v_{k},j,v^{(j)}_S,\ldots,v'_{t+1},a'_t,v'_t]$ give our desired result.
\item If $\mathrm{grd}(v^{(j)}_S)\geq \mathrm{grd}(v^{(i)}_L)$, then $\mathrm{grd}(v^{(i)}_S)< \mathrm{grd}(v^{(j)}_L)$ and there is an unbalanced edge $a_t$ in the walk $[a_{1}, a_2, \ldots,  a_{k-1}]$ satisfying $\mathrm{grd}(v_{t})<\mathrm{grd}(v_{t+1})$. Therefore the unbalanced edge $a_t$ and the walk $[v_{1},a_{1},v_{2},\ldots, v_{t}, a_{t},v_{t+1}]$ give our desired result.
\end{enumerate}
\end{enumerate}
\end{proof}
	
\begin{Lem}\label{cycle-nondomestic}
Let $\overline{\mathrm{gr}(A)}=kQ/I_2$. If some edges in the unique cycle are unbalanced edges, then the cardinality of $\mathrm{Ba}(\overline{\mathrm{gr}(A)})$ is infinite and  $\mathrm{gr}(A)$ is not of polynomial growth.
\end{Lem}
\begin{proof}
If some edges in the unique cycle are unbalanced edges, by Lemma \ref{greater-than-1}, then there are at least two unbalanced edges $v^{(i)}_S\defeqi v^{(i)}_L$ and $v^{(j)}_S\defeqj v^{(j)}_L$ with $v^{(i)}_S\neq v^{(j)}_L$ in the unique cycle such that there is a walk $[v_{1},a_{1},v_{2},a_{2},v_{3},\ldots, v_{k-1}, a_{k-1},v_{k}]$ from $v^{(i)}_S$ to $v^{(j)}_L$ satisfying $i\neq a_1$ and $a_{k-1}=j$, where $v_{1}=v^{(i)}_S$, $v_{k}=v^{(j)}_L$ and $a_{l}$ is an edge in the unique cycle incident to the vertices $v_{l}$ and $v_{l+1}$ for each $1\leq l\leq k-1$. Then $Q$ contains the following subquiver		
$$
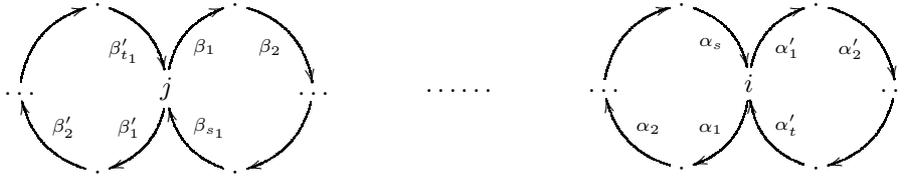
\xymatrix@C-0.8em@R-0.5em@r{
			&\cdot \ar@/^0.63pc/[rd]_{\beta_{t_{1}}'} &                                  &  \cdot\ar@/^0.635pc/[rd]_{\beta_{2}}    &  \\
			\ldots\ar@/^0.635pc/[ru]    &       &   j\ar@/^0.63pc/[ru]_{\beta_{1}}\ar@/^0.63pc/[ld]_{\beta_{1}'} &    &\ldots\ar@/^0.63pc/[ld]\\
			&           \cdot\ar@/^0.63pc/[lu]_{\beta_{2}'}     &        & \cdot\ar@/^0.63pc/[lu]_{\beta_{s_{1}}} &
}
\xymatrix{& & \\
			& \ldots\ldots &}
\xymatrix@C-0.8em@R-0.5em@r{
			&\cdot \ar@/^0.63pc/[rd]_{\alpha_{s}} &                                  &  \cdot\ar@/^0.635pc/[rd]_{\alpha_{2}'}    &  \\
			\ldots\ar@/^0.635pc/[ru]    &       &   i\ar@/^0.63pc/[ru]_{\alpha_{1}'}\ar@/^0.63pc/[ld]_{\alpha_{1}} &    &\ldots\ar@/^0.63pc/[ld]\\
			&           \cdot\ar@/^0.63pc/[lu]_{\alpha_{2}}     &        & \cdot\ar@/^0.63pc/[lu]_{\alpha_{t}'} &\quad ,
}$$
where $s=\mathrm{val}(v^{(i)}_S)$, $t=\mathrm{val}(v^{(i)}_L)$,   $s_1=\mathrm{val}(v^{(j)}_S)$, $t_1=\mathrm{val}(v^{(j)}_L)$,   $\alpha_{t}'\ldots\alpha_{1}'$ and $\beta_{t_{1}}'\ldots\beta_{1}'$ are not in $I_2$.
		
Since there is a unique cycle in $G$, there is a band $b_1=\alpha'_t\ldots\alpha^{-1}_s$ in $\overline{A}$. Therefore $b_1$ is also a band in $\overline{\mathrm{gr}(A)}$.
		
There is a simple string $c_{k_1}\ldots c_2 c_1$ satisfying $c_1=\alpha^{-1}_s$ and $t(c_{k_1})=j$ and it is constructed from the walk $[a_1,a_2,\ldots,a_{k-1}]$.
\begin{enumerate}[$(1)$]
\item If $c_{k_1}$ is an inverse arrow (in other words, $c_{k_1}=\beta_{1}^{-1}$), then $\beta_{t_{1}}'\ldots\beta_{1}'c_{k_1}\ldots c_2 c_{1}\alpha_{t}'\cdots\alpha_{1}'$ is also a string. There exists a simple string $c_{k_2}'\ldots c'_2 c_{1}'$ satisfying $c_{1}'=\beta_{s_{1}}^{-1}$ and $t(c'_{k_2})=i$. Then	
$$b_2:=\alpha_{t}'\ldots\alpha_{1}'c_{k_2}'\ldots  c'_2 c_{1}'\beta_{t_{1}}'\ldots\beta_{1}'c_{k_1}\ldots c_2 c_{1}$$
is a band with source $i$.
\item If $c_{k_1}$ is an arrow (in other words, $c_{k_1}=\beta_{s_{1}}$), then ${(\beta_{1}')}^{-1}\ldots{(\beta_{t_{1}}')}^{-1}c_{k_1}\ldots c_2 c_1$ $\alpha_{t}'\ldots\alpha_{1}'$ is also a string. In this situation we can similarly get a band $b_2$ as in $(1)$.
\end{enumerate}

Using an approach similar to the proof of Lemma \ref{nondomestic}, we have that the cardinality of $\mathrm{Ba}(\overline{\mathrm{gr}(A)})$ is infinite and $\mathrm{gr}(A)$ is not of polynomial growth.
\end{proof}
	
We give an example to illustrate the above result.
	
\begin{Ex}	
Let $G$ be the following Brauer graph with $m\equiv1$.	
$$
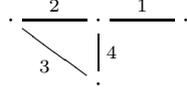
\xymatrix@R-1em{
			\cdot \ar@{-}[dr]_{3}\ar@{-}[r]^{2}& \cdot	\ar@{-}[r]^{1} \ar@{-}[d]^{4} &\cdot  \\
			&\cdot&	
}$$	
\\
Let $A=kQ/I$ be the Brauer graph algebra associated with $G$ and $\mathrm{gr}(A)$ the associated graded algebra of $A$. The quiver $Q$ is as follows.	
$$
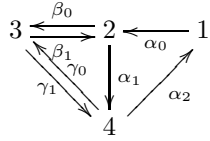
\xymatrix@r{
			3\ar@{->}@<-2pt>[r]_{\beta_1}\ar@{->}@<-2pt>[dr]^{\gamma_0}& 2\ar@{->}@<-2pt>[l]_{\beta_0}\ar@{->}[d]^{\alpha_1} &1\ar@{->}[l]^{\alpha_0}\\
			& 4\ar@{->}@<-2pt>[ur]_{\alpha_2}\ar@{->}@<-2pt>[ul]^{\gamma_1} &
}$$
Note that $\overline{\mathrm{gr}(A)}=\mathrm{gr}(A)/\soc(P_3)$, where $P_3$ is the projective cover of simple $\mathrm{gr}(A)$-module $S_3$ corresponding to the vertex $3$ in $Q$.
		
		We have that $b_1=\alpha_0\alpha_2\gamma^{-1}_0\beta_0 \alpha^{-1}_1\gamma_1\beta^{-1}_1$ and $b_2=\alpha_0\alpha_2\alpha_1\beta^{-1}_0\gamma_0\alpha^{-1}_2\alpha^{-1}_0\alpha^{-1}_1\gamma_1\beta^{-1}_1$ are bands in $\overline{\mathrm{gr}(A)}$.
\end{Ex}
	
By Proposition \ref{domestic-cycle}, Proposition \ref{cycle-not-polynomial} and  Lemma \ref{cycle-nondomestic}, we have the following characterization of domestic representation type of $\mathrm{gr}(A)$.
	
\begin{Prop}\label{graph-domestic}
Let $A$ be the Brauer graph algebra associated with a Brauer graph $G$ and $\mathrm{gr}(A)$ the graded algebra associated with the radical filtration of $A$, where $G$ is a graph with a unique cycle and $m(v)=1$ for all $v\in V(G)$. Then the following are equivalent.
	\begin{enumerate}[$(1)$]
		\item $\mathrm{gr}(A)$ is of polynomial growth.
		\item $\mathrm{gr}(A)$ is domestic.
		\item $\mathrm{gr}(A)$ is $1$-domestic (resp. $2$-domestic) if the unique cycle is of odd length (resp. even length).
		\item There is no unbalanced edges in $G$ or all edges in the unique cycle are not unbalanced edges and the unique cycle is in $G_{i,L}$ for any unbalanced edge $v_S\defeqi v_L$ (In other words, $G$ satisfies $\star$-condition with respect to any unbalanced edge in $G$). 
		\item The cardinality of $\mathrm{Ba}(\overline{\mathrm{gr}(A)})$ is finite, where  $\overline{\mathrm{gr}(A)}$ is defined in $(2.2)$.
\end{enumerate}
\end{Prop}

We can describe when $\mathrm{gr}(A)$ is domestic from the graded degrees of vertices in $G$ point of view in the following	

\begin{Prop}\label{graph-another}
Let $A$ be the Brauer graph algebra associated with a Brauer graph $G$ and $\mathrm{gr}(A)$ the associated graded algebra of $A$, where $G$ is a graph with a unique cycle and $m\equiv 1$. Then $\mathrm{gr}(A)$ is domestic if and only if it satisfies the following conditions.
\begin{enumerate}[$(1)$]
	\item $\mathrm{grd}(u)=\mathrm{grd}(v)$ for any two distinct vertices $u$ and $v$ in the unique cycle.
	\item Any walk from any vertex in the unique cycle is degree decreasing.
\end{enumerate}
\end{Prop}
	
\begin{proof}
``$\Longrightarrow$" Since all edges in the unique cycle are not unbalanced edges, $\mathrm{grd}(u)=\mathrm{grd}(v)$ for any two vertices $u$ and $v$ in the unique cycle (hence the condition $(1)$ holds).

In order to verify the condition (2). We suppose, on the contrary that, there is a vertex $w$ in $G$ such that a walk $[v_{1},a_{1},v_{2},\ldots,$ $v_{k-1},a_{k-1},v_{k}]$ from $v$ to $w$ is not degree decreasing, where $v$ is a vertex in the unique cycle and $v_k=w$. In other words, there is an unbalanced edge $v_{i}\defeqt v_{i+1}$ with $\mathrm{grd}(v_{i})<\mathrm{grd}(v_{i+1})$ for some $1\leq i\leq k-1$ in the walk. We have $v$ is in $G_{a_i,S}$ and the unique cycle is in $G_{a_i,S}$. Moreover, since $\mathrm{gr}(A)$ is domestic, by Proposition \ref{graph-domestic}, the unique cycle is in $G_{a_i,L}$. A contradiction.

``$\Longleftarrow$" We suppose on the contrary that $\mathrm{gr}(A)$ is nondomestic. By Proposition \ref{isomorphism} and Proposition \ref{graph-domestic}, since it contradicts the condition $(1)$ that some edges in the unique cycle are unbalanced edges, we have that all edges in the unique cycle are not unbalanced edges and therefore there is some unbalanced edge $v_S\defeqi v_L$ such that the unique cycle is in $G_{i,S}$.
For a vertex $v$ in the unique cycle and any walk $[v_{1},a_{1},v_{2},\ldots,$ $v_{k-1},a_{k-1},v_{k}]$ from $v$ to $v_L$ which is degree decreasing, we have that $i$ is an edge in the walk and therefore $\mathrm{grd}(v_S)\geq \mathrm{grd}(v_L)$, which is clearly a contradiction. Our assumption is false and therefore $\mathrm{gr}(A)$ is domestic.
\end{proof}

\subsection*{Acknowledgements} The first author is supported by Shandong Provincial Natural Science Foundation of China (No.ZR2022QA106) and China Scholarship Council. The second author is supported by the National Natural Science Foundation of China (No.12031014). The third author is partially supported by the National Natural Science Foundation of China (Grant Nos.12131015, 12161141001 and 12371042) and the Innovation Program for Quantum Science and Technology (Grant No.2021ZD0302902). We would like to thank Steffen Koenig for comments and many suggestions on the presentation of this paper. We would like to thank Yang Han for comments and useful discussions on our results. The first author would like to thank the representation theory group at the University of Stuttgart for hospitality during her study in Germany.		


\normalsize
	\begin{thebibliography}{50}
		\bibitem{BR}{{M. Butler} and {C. Ringel}, Auslander-Reiten sequences with few middle terms and applications to string algebras. Comm. Alg. {\bf 15} (1987), 145-179.}
	
		\bibitem{BS}{{R. Bocian} and {A. Skowro\'{n}ski}, Symmetric special biserial algebras of Euclidean type. Colloq. Math. {\bf 96} (2003), 121-148.}

		\bibitem{CPS}{{E. Cline}, {B. Parshall} and {L. Scott}, Graded and non-graded Kazhdan-Lusztig theories. In: Australian Math. Soc. Lecture Series, Algebraic groups and related subjects, a volume in honor of R. W. Richardson, 1996, pp. 100-120.}

		\bibitem{E}{{K. Erdmann}, {\it Blocks of tame representation type and related algebras}. Lecture Notes in Mathematics 1428, Springer, Berlin, 1990.}
	
		\bibitem{ES}{{K. Erdmann and A. Skowro\'{n}ski}, On Auslander-Reiten components of blocks  and self-injective biserial algebras. Trans. Amer. Math, Soc. {\bf 330} (1992), 165-189.}
	
		\bibitem{Ga}{{P. Gabriel}, Finite representation type is open. In: Representations of Algebras, Lecture Notes in Mathematics 488, Berlin-Heidelberg-New York, 1975, pp. 132-155.}

        \bibitem{Ge}{{C. Geiss}, On degenerations of tame and wild algebras. Arch. Math. {\bf 64} (1995), 11-16.}


		\bibitem{HPR}{{D. Happel}, {U. Preiser} and {C. M. Ringel}, Vinberg's characterization of Dynkin diagrams using subadditive functions with application to DTr-periodic modules. In: Proc. ICRA II (Ottawa 1979), Lecture Notes in Mathematics 832, Springer-Verlag, Berlin and New York, 1980, pp. 280-294.}
		
		\bibitem{GL}{{J. Guo and Y. Liu}, The associated graded algebras of Brauer graph algebras I: finite representation type. Comm. Algebra {\bf 49} (2021), 1071-1103.}

		
		\bibitem{RR}{{J. Rickard} and {R. Rouquier}, Stable categories and reconstruction. J. Algebra {\bf 475} (2017), 287-307.}

		\bibitem{S}{{S. Schroll}, Brauer Graph Algebras. In: I. Assem, S. Trepode (eds), Homological methods, representation theory, and cluster algebras. CRM Short Courses. Springer, 2018, pp. 177-223.}
		
\bibitem{SS}{{D. Simson} and {A. Skowro\'{n}ski}, {\it Elements of the representation theory of associative algebras 3: representation-infinite tilted algebras}. Cambridge University Press and Beijing World Publishing Corporation, 2010.}

\bibitem{S1}{{A. Skowro\'{n}ski}, Group algebras of polynomial growth. Manuscripta Math. {\bf 59} (1987), 499-516. }

		
        \bibitem{WW}{{B. Wald} and {J. Waschb\"{u}sch}, Tame biserial algebras. J. Algebra {\bf 95(2)} (1985), 480-500.}
	\end{thebibliography}
\end{document}